\documentclass[10pt]{amsart}
\usepackage{amssymb,amsmath,amsthm,graphicx,verbatim,amsbsy}
\usepackage{hyperref}

\newcommand{\Z}{{\bf Z}}
\newcommand{\R}{{\bf R}}

\newcommand{\LR}{{L^2({\bf R})}}
\newcommand{\Ltwoloc}{L^2_{\textrm{loc}}(J)}
\newcommand{\ip}[1]{\langle #1 \rangle}

\newcommand{\newterm}[1]{\textbf{#1}}

\newcommand{\spam}{\text{ span }}

\theoremstyle{plain}
\newtheorem{theorem}{Theorem}
\newtheorem{lemma}[theorem]{Lemma}

\newtheorem{corollary}[theorem]{Corollary}
\newtheorem{proposition}[theorem]{Proposition}

\theoremstyle{definition}

\theoremstyle{remark}
\newtheorem{rem}{Remark}

\title[Wavelets centered on a knot sequence]{Wavelets centered on a knot sequence: theory, construction, and applications}
\author[B. W. Atkinson, D. O. Bruff, J. S. Geronimo, D. P. Hardin]{Bruce W. 
Atkinson, Derek O. Bruff,\\
 Jeffrey S. Geronimo, and Douglas P. Hardin}

\thanks{The research of JSG was partially supported by NSF grant DMS-0500641 and the research of DPH was partially supported  by NSF grant  DMS-1109266.}
\newcommand{\ctable}{
\begin{table}[htdp] 
\caption{Scaling Coefficients for $\Phi$ \label{cTable}}
\begin{center}

\begin{align*}
C_{0,0}&=\left(
\begin{array}{lll}
 \sqrt{-\frac{1}{2}+\frac{\sqrt{5}}{2}} & -\frac{1}{2
   \sqrt{16491+7375 \sqrt{5}}} & \frac{121}{2 \sqrt{5485+2441
   \sqrt{5}}} \\
 0 & \frac{1}{4} \sqrt{-818+370 \sqrt{5}} & -\sqrt{\frac{55}{2
   \left(405+181 \sqrt{5}\right)}} \\
 0 & \frac{251}{\sqrt{76350+33982 \sqrt{5}}} &
   \frac{59}{\sqrt{30410+13610 \sqrt{5}}}
\end{array}
\right)
\\
     C_{0,1}&=\left(
\begin{array}{lll}
 19 \sqrt{\frac{2}{9765+4387 \sqrt{5}}} & \frac{1}{2}
   \sqrt{\frac{11}{767+343 \sqrt{5}}} & \frac{19}{2}
   \sqrt{\frac{5}{25615+11463 \sqrt{5}}} \\
 \sqrt{\frac{2}{11} \left(-35+16 \sqrt{5}\right)} & \frac{1}{4}
   \left(13-5 \sqrt{5}\right) & \frac{5}{\sqrt{370+166 \sqrt{5}}} \\
 -\frac{1}{11} \sqrt{\frac{5922}{5}-\frac{2584}{\sqrt{5}}} &
   -\sqrt{-\frac{283}{8}+\frac{637}{8 \sqrt{5}}} & \frac{1}{44}
   \left(-129+53 \sqrt{5}\right)
\end{array}
\right)
\\
C_{1,0}&= \left(
\begin{array}{lll}
 0 & -\frac{1}{2} \sqrt{\frac{11}{1715+767 \sqrt{5}}} &
   \frac{11}{\sqrt{2} \left(35+17 \sqrt{5}\right)} \\
 0 & 0 & 0 \\
 0 & 0 & 0
\end{array}
\right) \\
     C_{1,1}&=\left(
\begin{array}{lll}
 -\sqrt{-\frac{2}{5}+\frac{1}{\sqrt{5}}} & \frac{1}{4}
   \sqrt{177-\frac{389}{\sqrt{5}}} & -\frac{1}{4} \sqrt{-89+41
   \sqrt{5}} \\
 0 & 0 & 0 \\
 0 & 0 & 0
\end{array}
\right)\\
     C_{1,\tau}&=\left(
\begin{array}{lll}
 3-\sqrt{5} & 0 & 0 \\
 0 & 1 & 0 \\
 0 & 0 & 1
\end{array}
\right)
\\
C_{\tau,\tau^2}&=\left(
\begin{array}{lll}
 \frac{1}{2} \sqrt{\frac{1}{2} \left(13-3 \sqrt{5}\right)} &
   -\frac{1}{4} \sqrt{\frac{5}{36875+16491 \sqrt{5}}} &
   \frac{121}{4 \sqrt{2441+1097 \sqrt{5}}} \\
 0 & \sqrt{-\frac{409}{8}+\frac{185 \sqrt{5}}{8}} &
   -\sqrt{\frac{55}{2 \left(405+181 \sqrt{5}\right)}} \\
 0 & \frac{251}{\sqrt{76350+33982 \sqrt{5}}} &
   \frac{59}{\sqrt{30410+13610 \sqrt{5}}}
\end{array}
\right)\\
C_{\tau,\tau^2+1}&=\left(
\begin{array}{lll}
 \frac{1}{44} \left(-63+31 \sqrt{5}\right) & \frac{1}{4}
   \sqrt{\frac{55}{1715+767 \sqrt{5}}} & \frac{19}{4}
   \sqrt{\frac{5}{11463+5123 \sqrt{5}}} \\
 \sqrt{\frac{2}{11} \left(-35+16 \sqrt{5}\right)} & \frac{1}{4}
   \left(13-5 \sqrt{5}\right) & \frac{5}{\sqrt{370+166 \sqrt{5}}}
   \\
 -\frac{1}{11} \sqrt{\frac{5922}{5}-\frac{2584}{\sqrt{5}}} &
   -\sqrt{-\frac{283}{8}+\frac{637}{8 \sqrt{5}}} & \frac{1}{44}
   \left(-129+53 \sqrt{5}\right)
\end{array}
\right)
\\
C_{\tau^2,\tau^2}&=
\ \left(
\begin{array}{lll}
 0 & -\frac{1}{2} \sqrt{\frac{11}{2 \left(1030+461 \sqrt{5}\right)}}
   & \frac{1}{2} \sqrt{\frac{1}{310} \left(166-69 \sqrt{5}\right)}
   \\
 0 & 0 & 0 \\
 0 & 0 & 0
\end{array}
\right)
\\
C_{\tau^2,\tau^2+1}&=\left(
\begin{array}{lll}
 -\sqrt{\frac{1}{310}+\frac{\sqrt{5}}{62}} & \frac{1}{2}
   \sqrt{\frac{1}{310} \left(625-192 \sqrt{5}\right)} &
   -\sqrt{-\frac{37}{248}+\frac{4 \sqrt{5}}{31}} \\
 0 & 0 & 0 \\
 0 & 0 & 0
\end{array}
\right)
\\
C_{\tau^2,\tau^3}&=\left(
\begin{array}{lll}
 2 \sqrt{\frac{1}{31} \left(6-\sqrt{5}\right)} & -\frac{1}{2}
   \sqrt{\frac{5}{2 \left(93850+41971 \sqrt{5}\right)}} &
   \frac{121}{2 \sqrt{12492+5554 \sqrt{5}}} \\
 0 & \sqrt{-\frac{409}{8}+\frac{185 \sqrt{5}}{8}} &
   -\sqrt{\frac{55}{2 \left(405+181 \sqrt{5}\right)}} \\
 0 & \frac{251}{\sqrt{76350+33982 \sqrt{5}}} &
   \frac{59}{\sqrt{30410+13610 \sqrt{5}}}
\end{array}
\right)
\\
C_{\tau^2,\tau^3+1}&= \left(
\begin{array}{lll}
 19 \sqrt{\frac{2}{22219+9991 \sqrt{5}}} & \frac{1}{2}
   \sqrt{\frac{55}{2 \left(4365+1952 \sqrt{5}\right)}} &
   \frac{19}{2} \sqrt{\frac{5}{58306+26096 \sqrt{5}}} \\
 \sqrt{\frac{2}{11} \left(-35+16 \sqrt{5}\right)} & \frac{1}{4}
   \left(13-5 \sqrt{5}\right) & \frac{5}{\sqrt{370+166 \sqrt{5}}} \\
 -\frac{1}{11} \sqrt{\frac{5922}{5}-\frac{2584}{\sqrt{5}}} &
   -\sqrt{-\frac{283}{8}+\frac{637}{8 \sqrt{5}}} & \frac{1}{44}
   \left(-129+53 \sqrt{5}\right)
\end{array}
\right)
\end{align*}

\end{center}
\end{table}%
}

\allowdisplaybreaks
\newcommand{\dtable}{
\begin{table}[htdp] 
\caption{Wavelet Coefficients for $\Psi$ \label{dTable}}
\begin{center}
\begin{align*}
D_{0,0}&=\left(
\begin{array}{ccc}
 0 & \frac{1}{2 \sqrt{300+134 \sqrt{5}}} & -\frac{1}{2}
   \sqrt{\frac{1}{10} \left(-2+\sqrt{5}\right)} \\
 \frac{1}{2} \left(-1+\sqrt{5}\right) & \frac{1}{2 \sqrt{2
   \left(5096+2279 \sqrt{5}\right)}} & -\frac{1}{2}
   \sqrt{-168+\frac{761}{2 \sqrt{5}}}
\end{array}
\right)
\\
     D_{0,1}&=\left(
\begin{array}{ccc}
 \frac{1}{55} \left(-5+18 \sqrt{5}\right) & -\frac{1}{2}
   \sqrt{-\frac{1}{2}+\frac{2}{\sqrt{5}}} & -\frac{6+7 \sqrt{5}}{22
   \sqrt{2}} \\
 -\frac{1}{11} \sqrt{\frac{1}{10} \left(6085-2689 \sqrt{5}\right)}
   & -\frac{1}{2} \sqrt{\frac{11}{2 \left(237+106 \sqrt{5}\right)}}
   & -\frac{19}{2} \sqrt{\frac{5}{2 \left(7925+3538
   \sqrt{5}\right)}}
\end{array}
\right)
\\
D_{1,0}&= \left(
\begin{array}{lll}
 0 & \frac{1}{\sqrt{445+199 \sqrt{5}}} & -\sqrt{\frac{11}{5
   \left(69+31 \sqrt{5}\right)}}
\end{array}
\right) \\
     D_{1,1}&=\left(
\begin{array}{lll}
 \sqrt{\frac{2}{55} \left(13-5 \sqrt{5}\right)} & -\frac{1}{2}
   \sqrt{-45+\frac{103}{\sqrt{5}}} & \frac{1}{2} \sqrt{\frac{1}{11}
   \left(263-113 \sqrt{5}\right)}
\end{array}
\right)\\
     D_{1,\tau}&=\left(
\begin{array}{lll}
 \sqrt{-13+6 \sqrt{5}} & 0 & 0
\end{array}
\right)
\\
D_{\tau,\tau^2}&= \left(
\begin{array}{lll}
 \sqrt{-\frac{5}{8}+\frac{3 \sqrt{5}}{8}} & \frac{1}{4}
   \sqrt{\frac{31}{62071+27759 \sqrt{5}}} & -\frac{1}{8}
   \sqrt{4141-\frac{9161}{\sqrt{5}}} \\
 0 & \frac{1}{2 \sqrt{300+134 \sqrt{5}}} & -\frac{1}{2}
   \sqrt{\frac{1}{10} \left(-2+\sqrt{5}\right)}
\end{array}
\right)
\\
\end{align*}
\end{center}
\end{table}%
\begin{table*}[htdp] 
\begin{center}
{\normalsize Table~\ref{dTable} continued}
\begin{align*}
D_{\tau,\tau^2+1}&= \left(
\begin{array}{lll}
 -\sqrt{-\frac{7331}{968}+\frac{16557}{968 \sqrt{5}}} &
   -\frac{1}{4} \sqrt{\frac{341}{2887+1291 \sqrt{5}}} &
   -\frac{19}{4} \sqrt{\frac{155}{96375+43163 \sqrt{5}}} \\
 \frac{1}{55} \left(-5+18 \sqrt{5}\right) & -\frac{1}{2}
   \sqrt{-\frac{1}{2}+\frac{2}{\sqrt{5}}} & -\frac{6+7 \sqrt{5}}{22
   \sqrt{2}}
\end{array}
\right)
\\
D_{\tau^2,\tau^2}&=
\left(
\begin{array}{ccc}
 0 & \sqrt{\frac{873}{62}-\frac{976}{31 \sqrt{5}}} &
   -\frac{11}{\sqrt{6770+3040 \sqrt{5}}} \\
 0 & \frac{1}{2 \sqrt{199+89 \sqrt{5}}} & -\frac{1}{4}
   \sqrt{31-\frac{69}{\sqrt{5}}} \\
 0 & 0 & 0
\end{array}
\right)
\\
D_{\tau^2,\tau^2+1}&=
\left(
\begin{array}{ccc}
 \sqrt{\frac{2}{155} \left(23-9 \sqrt{5}\right)} &
   -\sqrt{-\frac{221}{31}+\frac{1009}{62 \sqrt{5}}} &
   \sqrt{\frac{1}{62} \left(234-101 \sqrt{5}\right)} \\
 \sqrt{-\frac{5}{22}+\frac{13}{22 \sqrt{5}}} & \frac{9-5
   \sqrt{5}}{4 \sqrt{2}} & \sqrt{-\frac{565}{176}+\frac{263
   \sqrt{5}}{176}} \\
 0 & 0 & 0
\end{array}
\right)
\\
D_{\tau^2,\tau^3}&=\left(
\begin{array}{ccc}
 \sqrt{\frac{1}{31} \left(7+4 \sqrt{5}\right)} & \sqrt{\frac{5}{2
   \left(397555+177792 \sqrt{5}\right)}} &
   -\frac{121}{\sqrt{52754+23600 \sqrt{5}}} \\
 0 & -\frac{1}{2 \sqrt{9349+4181 \sqrt{5}}} & \frac{121}{2}
   \sqrt{\frac{11}{34145+15249 \sqrt{5}}} \\
 0 & \frac{1}{2 \sqrt{300+134 \sqrt{5}}} & -\frac{1}{2}
   \sqrt{\frac{1}{10} \left(-2+\sqrt{5}\right)}
\end{array}
\right)\\
D_{\tau^2,\tau^3+1}&= \left(
\begin{array}{ccc}
 -38 \sqrt{\frac{2}{94393+42201 \sqrt{5}}} & -\sqrt{\frac{55}{2
   \left(18490+8269 \sqrt{5}\right)}} & -19 \sqrt{\frac{5}{2
   \left(123546+55249 \sqrt{5}\right)}} \\
 19 \sqrt{\frac{2}{5545+2483 \sqrt{5}}} & \frac{11}{2
   \sqrt{4783+2139 \sqrt{5}}} & \frac{19}{2}
   \sqrt{\frac{5}{14525+6497 \sqrt{5}}} \\
 \frac{1}{55} \left(-5+18 \sqrt{5}\right) & -\frac{1}{2}
   \sqrt{-\frac{1}{2}+\frac{2}{\sqrt{5}}} & -\frac{6+7 \sqrt{5}}{22
   \sqrt{2}}
\end{array}
\right)
\end{align*}
\end{center}
\end{table*}%
}


\begin{document}

\begin{abstract}
We develop a general notion of orthogonal wavelets `centered' on an
irregular knot sequence. We present two families of orthogonal wavelets
that are continuous and piecewise polynomial. We develop efficient
algorithms to implement these schemes and apply them to a data set
extracted from an ocelot image.  As another application, we construct continuous, piecewise quadratic, orthogonal wavelet bases on the quasi-crystal lattice consisting of the  $\tau$-integers where $\tau$ is the golden ratio.  The resulting  spaces then generate a multiresolution analysis  of $L^2(\mathbf{R})$ with scaling factor $\tau$.  
\end{abstract}

\maketitle

\section* {Introduction}\label{intro}

Wavelets are a useful tool for representing general functions and
datasets and are now used in a wide variety of areas such as signal
processing, image compression, function approximation, and 
finite-element methods. Traditionally wavelets are constructed from one
function, the so-called mother wavelet, by integer translation and dyadic
dilation and give rise to a stationary refinement equation.  For many functions however this may not
be the most efficient way of approximating them and one may require
adaptive schemes for which the knot sequence is not regularly
spaced or where the knots are on a regular grid but the refinement equation
changes at each step. There has been much work devoted to extending wavelet
constructions to irregularly spaced knot sequences including the lifting scheme of
Sweldens \cite{S} and wavelets on irregular grids (see Daubechies
et. al. \cite{DGSS},\cite{DGS}, Charina and St\"{o}ckler \cite{ChSt}) or  nonstationary tight wavelet frames (see Chui, He, and St\"ockler, \cite{Chui1, Chui2} and Shah \cite{Shah}) or to nonstationary refinement masks (see Cohen and 
Dyn \cite{Dy} and Herley, Kova\u{c}evi\'{c}, and Vetterli \cite{HKV}).   If the knots are allowed to be chosen generally enough
it is difficult to maintain the orthogonality  of the wavelet functions
while keeping the compact support and smoothness properties.  

In \cite{DGHinter, DGHorpoly}   piecewise polynomial, orthogonal, multiwavelets with compact support were constructed using multiresolution analyses 
that were {\em intertwined} with classical spline spaces.  Both the multiwavelets as well as the scaling functions  generating these multiwavelets have `short' support  allowing
them to be adapted to irregular knot sequences using the machinery of  {\em squeezable bases} developed in \cite{DGHspie, DGHsqueeze} in a way that preserved the orthogonality, polynomial reproduction, and smoothness of the respective bases.  In general, the 
multiresolution analyses were restricted to  {\em semi-regular} refinement schemes in which an initial irregular knot sequence is refined by a regular refinement scheme such as mid-point subdivision (cf.~\cite{BH}).  However, in \cite{DGHsqueeze}, a construction  was given of a fully irregular multiresolution analysis consisting of continuous, piecewise quadratic functions on  an {\em arbitrary} sequence of nested knot sequences such that the multiresolution spaces have compactly supported orthogonal bases with `short' support (we review and generalize this construction in Section~\ref{ex1}).  The resulting spaces in the multiresolution analysis did not fit, for general knot sequences, into the framework   of squeezable bases.    In \cite{BruffThesis},  a more general notion of {\em bases centered on a knot sequence} was introduced that includes the squeezable bases as well as the irregular construction. In particular, necessary and sufficient conditions were given for when a space $V$ has an {\em orthogonal} basis centered on a knot sequence.   This is used to prove one of the main  results  of this paper that if $V^0 \subset V^1$ are spaces generated by orthogonal bases centered on a common  knot sequence $\mathbf{a}$ then the orthogonal complement $W=V^1\ominus V^0$ is also generated by  an orthogonal basis $\Psi$ centered on $\mathbf{a}$.   

In Section~\ref{sect:GeneralBases} we review and further elucidate the theory of bases centered on a knot sequence.  In particular, as mentioned above, we focus on the 	question of characterizing spaces that are generated by orthogonal bases centered on a knot sequence. We next give two different constructions of nested spaces of this type. The first uses a fixed knot sequence $\mathbf{a}$ and constructs $V^0 \subset V^1 \subset \cdots$ which are spaces generated by orthogonal continuous piecewise polynomial functions with compact support and with breakpoints in $\mathbf{a}$; these spaces have increasing polynomial reproduction. The second construction begins with a given knot sequence $\mathbf{a}$ and the spline space $S^0_2(\mathbf{a})$ consisting of continuous piecewise quadratic functions with breakpoints in $\mathbf{a}$. The knot sequence $\mathbf{a}$ is then refined and a space $V$, containing $S^0_2(\mathbf{a})$,   is constructed which is generated by continuous orthogonal piecewise quadratic functions with compact support and breakpoints in the refined sequence. If  $\mathbf{a^0}\subset \mathbf{a^1}$ and $V^0$ (resp. $V^1$) is constructed in this way starting with $\mathbf{a^0}$ (resp. $\mathbf{a^1}$), we describe general conditions under which $V^0 \subset V^1$. 
 
In Section~\ref{Wavelets} it is shown how to build wavelets from the scaling functions constructed in the previous section. Certain spaces are introduced which shed  light on techniques of  \cite{DGHorpoly}. The wavelet construction methods   given in  \cite{DGHorpoly} are sufficiently general so that they can be used to give a decomposition of the wavelet spaces,
thus providing a convenient algorithm for calculating these bases. 

In Section~\ref{efficient_constr} 
efficient
algorithms are developed in order to demonstrate practicality of the method
developed in the previous sections.
The matrices, defined in section~\ref{Wavelets},  in the equations
for the scaling functions and wavelets are computed after a knot is added or
dropped. These algorithms are based upon the greedy algorithm and nonlinear
approximation schemes \cite{BCDD}, \cite{Dy}, \cite{DFI} and an application
to a data set extracted from the digital image of an ocelot is given to show their effectiveness.

In Section \ref{tauwavelets} a
construction of  multiwavelets is
carried out for the knot sequence consisting of  the
$\tau$-integers, where $\tau=\frac{1}{2} (1+\sqrt 5)$
is the golden mean.  These multiwavelets have a scaling factor $\tau$
and will be called $\tau$-multiwavelets.  The lattice generated by 
$\tau$ and other Pisot numbers appears in the study of 
quasi-crystals and powers of these numbers appear in the diffraction
patterns of actual experiments.  
$\tau$-Haar wavelets were constructed in \cite{GP, GS}. $\tau$-Haar wavelets are orthogonal and compactly supported but they
are not continuous. The above construction is used to give examples of piecewise quadratic \emph{continuous} compactly
supported $\tau$-multiwavelets.     Additional work on multiresolution analyses with irrational scaling factors includes 
  Chui and Shi \cite{CS}, Hern\'{a}ndez, Wang, and Weiss \cite{HWWI}, \cite{HWWII}, and Bownik \cite{Bownik}.



\section{Bases Centered on Knot Sequences}

\label{sect:GeneralBases}


Let $J$ be an interval in $\mathbf{R}$ and let $\mathbf{a}\subset J$ have no
cluster point in $J$ and such that $\inf\mathbf{a}=\inf J$ and $\sup \mathbf{a}=\sup J$; 
to avoid trivial cases we assume that  $\mathbf{a}$ consists of at least three elements. We refer to
such a set $\mathbf{a}$ as a \textbf{knot sequence} in $J$ since it can be
represented as the range of a strictly increasing sequence indexed by an
interval $\mathcal{I}\subset\mathbf{Z}$.

Let $\mathbf{a}$ be a knot sequence
in $J$. If $a\in\mathbf{a}$, we define $a_{+}:=\inf
\{b\in\mathbf{a}|a<b\}$; in this case we shall refer to $a_{+}$ as the
\textbf{successor} of $a$  in $\mathbf{a}$.
  We will write $a_{++}$ for $(a_{+})_{+}$.
Similarly,   we define $a_{-}:=\sup\{b\in\mathbf{a}|b<a\}$;
i.e. $a_{-}$ is the \textbf{predecessor} of $a$ in $\mathbf{a}$.   We remark that
$a_+=\infty$ if  $a=\sup J$ and $a_-=-\infty$ if  $a=\inf J$.

For an arbitrary collection of functions  $G\subset L^{2}(J)$ and an interval $I\subset J$, 
 let
  $$G_I:=\{g\in G \ |\ \text{supp }g\subset I\}.$$ 
If  $\Phi\subset L^{2}(J)$ is a locally finite collection of functions    (i.e.,
on any compact interval  $K\subset J$ all but a finite number of $\phi\in \Phi$  vanish on $K$)   then for $a\in\mathbf{a}$ we define
\begin{equation} \label{suba}
\begin{split}  
\breve{\Phi}_{a}&:=\Phi_{[a,a_+]},  \\ 
\Phi_{a}&:=\Phi_{[a_{-},a_+]}\setminus\breve{\Phi}_{a_{-}}, \text{ and}\\
\bar{\Phi}_{a}&:=\Phi_{a}\setminus\breve{\Phi}_{a},  
\end{split}
\end{equation}
where we define $\breve \Phi_{-\infty}:=\emptyset$ and $\bar\Phi_{\infty}:=\emptyset$.   We remark that if $J$ contains its supremum and $a=\sup J$ then $\Phi_a=\emptyset$ while if $J$ contains its infimum and $a=\inf J$ then
$\bar \Phi_a=\emptyset$.  The sets $\breve{\Phi}_a$, $\bar{\Phi}_a$, and $\Phi_a$ depend on the knot sequence  ${\mathbf a}$ since $a_+$ and $a_-$ are defined relative to $\mathbf{a}$.  When there is a chance of ambiguity we write $\breve{\Phi}_{a,\mathbf{a}}$, $\bar{\Phi}_{a,\mathbf{a}}$, and $\Phi_{a,\mathbf{a}}$
to denote the knot sequence $\mathbf{a}$ that is referenced. 
 
We say that 
$\Phi\subset L^2(J)$ is  a \textbf{basis \ centered on the
knot sequence }$\mathbf{a}$ provided 
 \begin{equation}\label{Phi:Def}
 \begin{split}
\text{(a) }& \text{$\Phi$ is locally finite,}\\
\text{(b) }& \text{ $\Phi=\bigcup_{a\in\mathbf{a}}\Phi_a$, and}\\
\text{(c) }& \text{ $\left(  \Phi_{a}\cup\bar{\Phi}_{a_{+}%
}\right)  |_{[a,a_{+}]}$ is a linearly independent set for all $a\in \mathbf{a}$}.
 \end{split}  
 \end{equation}  
 For any $\Phi$ satisfying  conditions (a) and (b) of \eqref{Phi:Def}, we let 
\[
\label{SPhi}
S(\Phi) := \text {clos}_{L^2(J)} \text{span }   (\Phi).\]
 
  The notion  of a basis centered on a knot sequence was introduced in \cite{BruffThesis} and  is a generalization of
bases obtained from \textbf{minimally supported generators }as defined in
\cite{DGHinter, DGHsqueeze}.   Roughly speaking,   a basis centered on a knot sequence consists of functions whose supports 
 overlap at most on a single `knot interval' and the non-zero restrictions of these functions to each such knot interval are linearly independent.


If $\Phi$ is a basis centered on $\mathbf{a}$  then it 
follows from  properties \eqref{Phi:Def}  that  any $f\in S(\Phi)$ has a unique representation of the form
\begin{equation}\label{basis_expansion}
 f=\sum_{a\in\mathbf{a}} c_a\Phi_a=\sum_{a\in\mathbf{a}} (\breve{c}_a\breve\Phi_a+
 \bar{c}_a\bar\Phi_a),
 \end{equation}
where the convergence of the sums is in   $\Ltwoloc$.  Note that $c_a$, $\breve{c}_a$, and $\bar{c}_a$
are treated as row vectors and $\Phi_a$,  $\breve{\Phi}_a$, and $\bar{\Phi}_a$ are treated as column vectors so that $c_a\Phi_a$ denotes the linear combination of elements $\Phi_a$ with
coefficient vector $c_a$; the expressions   
$\breve{c}_a\breve\Phi_a$ and
 $\bar{c}_a\bar\Phi_a$ are interpreted similarly.  It then follows from the local linear independence condition (c) that 
 \begin{equation}\label{coeffvanish1}
 \text{ $f\in S(\Phi)$ and $f=0$ on  $[a,a_+]$} \implies \text{
 $\bar c_a$, $\bar c_{a_+}$ and $\breve c_a$ are all 0}, 
 \end{equation}
 and then \eqref{coeffvanish1} implies 
\begin{equation}
 S(\Phi)_{[a,b]}=\text{ span } \left(\Phi_{[a,b]}\right)=\text{ span }(\breve{\Phi}%
_{a}\cup(\bigcup_{\substack{c\in\mathbf{a} \\ a<c<b}}\,\Phi_{c}))\label{spancond},%
\end{equation}
for $a,b\in\mathbf{a}$ with $a<b$. 
Note that if there are no knots $c$ between $a$ and $b$, then the sum (resp.
union) in equation (\ref{spancond}) is zero (resp. empty).  

For  $V\subset L^2(J)$,  let 
$$V_{a}:=V_{[a_{-},a_+]} \text{ and }\breve{V}_a:=V_{[a,a_+]}\qquad a\in \mathbf{a}.$$  
Note that if $\Phi$ is a basis centered on the knot sequence $\mathbf{a}$ and
$V=S(\Phi)$, then equation (\ref{spancond}) implies that for $a\in \mathbf{a}$,
\[\breve V_a=\text{span }\breve\Phi_a
\]
and 
\[
V_a= 
\text{span (}\breve{\Phi}_{a_{-}}\cup\Phi_{a}).
\]
In particular it follows that each $V_a$ is finite dimensional. 
 We next   characterize when  $V\subset L^2(J)$ equals $S(\Phi)$ for
 some basis $\Phi$ centered on $\mathbf{a}$. 
 
\begin{theorem}
\label{LemmaSpace2Basis}    Let $\mathbf{a}$ be a knot sequence in  an interval 
 $J$ and suppose   $V \subset L^2(J)$.
Then $V=S(\Phi)$ for some basis $\Phi$ centered on $\mathbf{a}$ if and only if
the following three properties hold.

\begin{itemize}
\item[($\hat a$)] $V_a$ is a finite dimensional subspace of  $L^2(J)$ for $a\in \mathbf{a}$.

\item[($\hat b$)] $V=\text {clos}_{L^2(J)}\spam\left(\bigcup_{a\in \mathbf{a}}V_a\right).$

\item[($\hat c$)] If $f\in V$ vanishes on $[a,a_{+}]$, then  $f\in V_{(-\infty,a ]}+V_{[a_+,\infty )}$.
\end{itemize}

If ($\hat a$--$\hat c$) hold, then, for $a\in\mathbf{a}$, let $\breve\Phi_a$ be a basis for $\breve V_a$ 
and let $\bar\Phi_a$ augment $\breve\Phi_{a_{-}}\cup\breve\Phi_a$ to a basis of
$V_a$.   Then $\Phi:= \bigcup_{a\in\mathbf{a}}\left(\bar \Phi_{a} \cup\breve \Phi_a\right)$ is a basis centered on $\mathbf{a}$ such that $V=S(\Phi).$  Furthermore, any  $\Phi$ 
such that $V=S(\Phi)$ is of this form, that is,  $\breve\Phi_a$ must be a basis of $\breve V_a$ and $\bar\Phi_a\cup \breve\Phi_{a_{-}}\cup\breve\Phi_a$ must be a basis of
$V_a$.
\end{theorem}

\begin{rem}    Theorem~\ref{LemmaSpace2Basis} implies that if $\Phi$ is a basis centered on a knot sequence $\mathbf{b}$  that is a \newterm{refinement} of $\mathbf{a}$ (i.e., $\mathbf{a}\subset\mathbf{b}$), then $\Phi$ is also a basis centered on $\mathbf{a}$.  Note that, for  $a\in \mathbf{a}\subset \mathbf{b}$,  the sets $\Phi_{a, \mathbf{a}}$ and $ \Phi_{a,\mathbf{b}}$ are, in general, not equal.  
\end{rem}

\begin{proof}
($\Rightarrow$) Suppose $V=S(\Phi)$ for some basis $\Phi$ centered on the knot
sequence $\mathbf{a}$. Conditions ($\hat a$) and ($\hat b$) then follow from (\ref{spancond}) 
and condition ($\hat c$) follows from \eqref{coeffvanish1}. 
 
($\Leftarrow$) Suppose that ($\hat a$--$\hat c$) hold and that $\breve{\Phi}_{a}$ and
$\bar{\Phi}_{a}$ are constructed as in the statement of
Theorem~\ref{LemmaSpace2Basis}. Then ($\hat b$) implies $V=S(\Phi)$ and so it remains to show the local linear independence condition of $\Phi$. 
Let $a\in\mathbf{a}$ and suppose that $\bar c_a$, $\breve c_a$, and $\breve c_{a_+}$
are vectors so that 
\begin{equation}\label{SPhilemma.1}
f:=\bar c_a\bar \Phi_a+\breve c_a\breve \Phi_a+\bar c_{a_+}\bar \Phi_{a_+}=0 
\text{ on $[a,a_+]$.}
\end{equation}  It now suffices to show that these coefficient vectors must all
vanish.   By condition ($\hat c$), we have $f=f_1+f_2$ where $f_1\in V_{(-\infty,a ]}$
and $f_2\in V_{[a_+,\infty )}$.   Since $f$ and $f_2$ vanish on $(-\infty, a_{-}]$ it follows that 
$f_1\in V_{[a_{-},a]}$. Hence, $f_1=\breve d_{a_{-}}\breve\Phi_{a_{-}}$ for some vector  $\breve d_{a_{-}}$ by the construction of $\breve \Phi_{a_{-}}$.  Then  
$g:=\bar c_a\bar \Phi_a-\breve d_{a_{-}}\breve\Phi_{a_{-}}=0$ on $[a_{-},a]$. Since the support of $g$ is in $[a,a_+]$, then $g=\breve d_a\breve\Phi_a$ for some vector $\breve d_a$, i.e.,
$$  \bar c_a\bar \Phi_a-\breve d_{a_{-}}\breve\Phi_{a_{-}}-\breve d_a\breve\Phi_a=0.$$
Since  $\bar \Phi_{a}$,  $\breve\Phi_{a_{-}}$ and $\breve\Phi_{a }$ are linearly independent, 
it follows that $  \bar c_a=0$.  Similarly,  we have $\bar c_{a_+}=0$ and thus, by the linear
independence of $\breve\Phi_a$ we have $\breve c_a=0$. 
\end{proof}

\subsection{Example}

\label{sect:spline_example} Let $\Pi^{d}$ denote the space of univariate
polynomials of degree at most $d$. For a knot sequence $\mathbf{a}$ on an
interval $J$ and integers $r<d$, let
\[
S_d^r(\mathbf{a}):=\{f \mid   \text{supp } f \subset J,\,  f\in C^{r}(J)\cap L^2(J), \text{ and } f|_{(a,a_{+})}\in\Pi^{d}\text{
for }a\in\mathbf{a} \}
\]
denote the classical spline space of degree $d$ and regularity $r$ intersected with $L^2(J)$. If $d\geq 2r+1$, then  $S_d^r%
(\mathbf{a})=S(\Phi)$ for a basis $\Phi$ centered on
$\mathbf{a}$.  For example, the B-spline basis for $S_d^r(\mathbf{a})$, which we denote by $\Phi_{r,d}^{\mathbf{a}}$, is  
such a basis, cf.~\cite{Schumaker}. Furthermore,  $|\breve{\Phi}_{a}|=d-2r-1$ if
$a<\sup J$, and for $\inf J<a<\sup J$, we have $|\Phi_{a}|=d-r$ and  $|\bar{\Phi}_{a}|=r+1$.

\subsection{Orthogonality Condition}

\label{sect:GeneralOrthogonality}

Bases $\Phi$ and $\Psi$ centered on the same knot sequence $\mathbf{a}$ in the  interval $J$
are called \textbf{equivalent} if $S(\Phi)=S(\Psi)$. We next give a necessary
and sufficient condition for a basis $\Phi$ centered on $\mathbf{a}$ to be
equivalent to some \textit{orthogonal} basis $\Omega$ centered on $\mathbf{a}%
$. This condition is the main tool we use to construct orthogonal bases
centered on a knot sequence. Previous versions of this theorem appeared in
\cite{DGHinter} (for the shift-invariant setting),  in \cite{DGHsqueeze} (for
the setting of \textquotedblleft squeezable\textquotedblright\ orthogonal
bases) and in \cite{BruffThesis} (for the current setting of bases centered on a
knot sequence).


\begin{theorem}
\label{LemmaOrthCond} Let $J$ be an interval in $\mathbf{R}$, $\mathbf{a}$ a
knot sequence on $J$,  $\Phi$ a basis centered on $\mathbf{a}$, and 
$V=S(\Phi)$. Then there exists an orthogonal basis $\Omega$ centered on the
knot sequence $\mathbf{a}$ such that $S(\Omega)=S(\Phi)$ if and only if
\begin{equation}
\left(  I-P_{\breve{V}_{a}}\right)  V_{a}\perp V_{a_{+}}\text{ for } 
a\in \mathbf{a}\text{.}\label{EqnOrthCond}%
\end{equation}
If (\ref{EqnOrthCond}) holds, then, for $a\in \mathbf{a}$,   let $\breve{\Omega}_{a}$ be an orthogonal basis of $\breve{V}_{a}$,  $\bar{\Omega}_{a}$ be an orthogonal basis of $\left(  I-P_{\breve{V}_{a_{-}}\oplus\breve{V}_{a}}\right)V_a$, and $\Omega_a=\breve\Omega_a\cup \bar\Omega_a$.  Then
$\Omega=\{\Omega_{a}\}_{a\in\mathbf{a}}$ is such an orthogonal basis.
\end{theorem}


\begin{proof}

Suppose $a\in\mathbf{a}$.  Note that 
$\breve{V}_{a_{-}}\perp V_{a_{+}}$ since their supports intersect in at most one point.  It then follows that (\ref{EqnOrthCond}) is
equivalent to
\begin{equation}
\left(  I-P_{\breve{V}_{a_{-}}\oplus\breve{V}_{a}}\right)  V_{a}\perp
V_{a_{+}}\text{.}\nonumber
\end{equation}
Furthermore, since $\left(  I-P_{\breve{V}_{a_{-}}\oplus\breve{V}_{a}}\right)
V_{a}$ is orthogonal to $\breve{V}_{a}$ and, due to
support properties,  is also orthogonal to $\breve{V}_{a_{+}}$, it follows that (\ref{EqnOrthCond}) is equivalent to
\begin{equation}
\left(  I-P_{\breve{V}_{a_{-}}\oplus\breve{V}_{a}}\right)  V_{a}\perp\left(
I-P_{\breve{V}_{a}\oplus\breve{V}_{a_{+}}}\right)  V_{a_{+}}\text{.}
\label{EqnOrthCondAlt}%
\end{equation}

($\Rightarrow$) Suppose that $V=S(\Omega)$, where $\Omega$ is an orthogonal
basis centered on the knot sequence $\mathbf{a}$. It follows from
(\ref{spancond}) and the orthogonality of $\Omega$ that $$V_{a}=\text{span }
\breve{\Omega}_{a_{-}}\oplus \text{span }\bar{\Omega}_{a}\oplus\text{span }\breve
{\Omega}_{a}=\breve{V}_{a_{-}}\oplus\text{span }\bar{\Omega}_{a}\oplus\breve
{V}_{a},$$  and so $\text{span}(\bar{\Omega}_{a})=\left(
I-P_{\breve{V}_{a_{-}}\oplus\breve{V}_{a}}\right)  V_{a}$. Since $\Omega$ is
an orthogonal basis, $\text{span}(\bar{\Omega}_{a})\perp\text{span}%
(\bar{\Omega}_{a_{+}})$, and so
(\ref{EqnOrthCondAlt}), and hence (\ref{EqnOrthCond}), holds.

($\Leftarrow$) Suppose that $V$ satisfies (\ref{EqnOrthCond}), and hence
satisfies (\ref{EqnOrthCondAlt}). Construct  $\Omega$ as in the statement of
 Theorem~\ref{LemmaOrthCond}.  By Theorem~\ref{LemmaSpace2Basis},  it follows that   $\Omega$ is a basis
centered on $\mathbf{a}$ such that $V=S(\Omega)$. Equation  
(\ref{EqnOrthCondAlt}) implies that $\bar{\Omega}_{a}\perp\bar{\Omega}_{a_{+}}$ for $a\in\mathbf{a}$, which shows $\Omega$ is an orthogonal basis of $V$. 
\end{proof}

Suppose $V=S(\Phi)$ for some $\Phi$. Then it is
easy to verify that (\ref{EqnOrthCond}) holds if and only if
\begin{equation}
\left(  I-P_{\breve{V}_{a}}\right)  \bar{\Phi}_{a}\perp\bar{\Phi}_{a_{+}}, \qquad  a\in\mathbf{a}\text{.} \label{OrthBarEqn}%
\end{equation}
We use this equation as the basis for a construction of orthogonal bases
centered on a knot sequence as we next describe. Suppose $$Z:=\text {clos}_{L^2(J)} \text{span } \big( \bigcup_{a\in\mathbf{a}}Z_{a}\big ),$$ where  $Z_{a}$, $a\in\mathbf{a}$ is a
finite dimensional subspace of $L^{2}(J)$ such that (a) the elements
of $Z_{a}$ are supported in $[a,a_{+}]$ and (b) $Z_{a}$ is linearly
independent of $V$ restricted to $[a,a_{+}]$. Then $U:=V+Z$ satisfies the hypotheses ($\hat{a}$), ($\hat{b}$), and ($\hat{c}$) of
Theorem~\ref{LemmaSpace2Basis} and, by this lemma, $U=S(\Theta)$ where 
$\Theta$ is a basis centered on
$\mathbf{a}$ such that $\breve\Theta_a$ is a basis for  $\breve{U}_{a}=\breve{V}_{a}+Z_{a}$ and $\bar\Theta_a:=\bar\Phi_a$ for $a\in\mathbf{a}$.  By  Theorem~\ref{LemmaOrthCond}, $U=S(\Omega)$
for some orthogonal basis $\Omega$ centered on $\mathbf{a}$ if and only if
\begin{equation} 
\left(  I-P_{\breve{V}_{a}+Z_{a}}\right)  \bar{\Phi}_{a}\perp\bar{\Phi}_{a_{+}%
}, \qquad a\in\mathbf{a}. \label{OrthBarEqn2}%
\end{equation}

Without loss of generality, we may choose $Z_a$ orthogonal to $\breve{V}_{a}$. 
Then \eqref{OrthBarEqn2} is equivalent to
\begin{equation} 
\langle  (I-P_{\breve{V}_{a}} )  \bar{\Phi}_{a}, \bar{\Phi}_{a_{+}}\rangle%
 =\langle    P_{Z_a}\bar{\Phi}_{a}, \bar{\Phi}_{a_{+}}\rangle, \qquad a\in\mathbf{a}. \label{OrthBarEqn3}%
\end{equation}
where for finite collections $F,G\subset \LR$, we let $\langle F, G\rangle $ denote the
matrix $\left(\langle f, g\rangle\right)_{f\in F, g\in G}$ indexed by $F$ and $G$.  
Then from \eqref{OrthBarEqn3} it follows that  $$\dim Z_a\ge \text{rank }\langle  (I-P_{\breve{V}_{a}} )  \bar{\Phi}_{a}, \bar{\Phi}_{a_{+}}\rangle.$$
We remark  that if one finds a $Z_a$ that satisfies \eqref{OrthBarEqn2}, then one can always choose $Z'_a\subset Z_a$ such that equality holds in the above estimate.

In this paper, we focus on constructions in which  the spaces $Z_{a}$ above are
chosen via a generalization of the \textbf{intertwining technique} developed
in \cite{DGHinter} (and extended in \cite{BruffThesis} and \cite{DGHsqueeze}) to
construct orthogonal piecewise polynomial wavelets. To that end, suppose
$\Phi^{\prime}$ and $\Phi^{\prime\prime}$ are bases centered on $\mathbf{a}$
such that
\begin{equation}
S(\Phi)\subset S(\Phi^{\prime})\subset S(\Phi^{\prime\prime}) \label{intereq}%
\end{equation}
and that there exist spaces $Z_{a}\subset$ span $\breve{\Phi}_{a}^{\prime}$
(respectively, $Z_{a}^{\prime}\subset$ span $\breve{\Phi}_{a}^{\prime\prime}%
$), $ a\in\mathbf{a}$, as in the previous paragraph, so that there
exists an orthogonal basis $\Omega$ (respectively, $\Omega^{\prime}$) such
that $S(\Omega)=S(\Phi)+S(Z)$ (respectively, $S(\Omega^{\prime})=S(\Phi
^{\prime})+S(Z^{\prime})$). It then follows that
\[
S(\Omega)\subset S(\Phi^{\prime})\subset S(\Omega^{\prime}).
\]
Specifically, suppose $\mathbf{b}$ and $\mathbf{c}$ are knot sequences in $J$
so that $\mathbf{a}\subset\mathbf{b}\subset\mathbf{c}$ and consider the spline
bases $\Phi=\Phi_{r,d}^{\mathbf{a}}$, $\Phi^{\prime}=\Phi_{r,d}^{\mathbf{b}}$,
and $\Phi^{\prime\prime}=\Phi_{r,d}^{\mathbf{c}}$ where $r\geq2d+1$. Then
$S(\Phi)\subset S(\Phi^{\prime})\subset S(\Phi^{\prime\prime})$. We also
recall that  $\Phi^{\prime}$ and $\Phi
^{\prime\prime}$ are bases centered on $\mathbf{a}$ as well as centered on
$\mathbf{b}$ and $\mathbf{c}$, respectively; see the remark following
Theorem~\ref{LemmaSpace2Basis}.

\subsection{Continuous, orthogonal, spline basis centered on a knot sequence: arbitrary polynomial reproduction.
\label{ex1.2}} In this example we will use extensively the results of
\cite{DGHorpoly}. Let $J$ be an interval in $\mathbf{R}$, $\mathbf{a}$ a knot sequence in $J$,
and let $S_{n}^{0}(\mathbf{a})$ denote the spline space consisting of
continuous piecewise polynomial functions in $L^2(J)$ of degree at most  $n$ defined on $J~$with break
points in $\mathbf{a}$. Then $S_{n}^{0}(\mathbf{a})=S(\Phi)$ for a basis
$\Phi$ centered on $\mathbf{a}$. Let $I=[0,1]$, and let $\Pi^n$ denote the space of polynomials
of degree $n$ on $I$. If we set
$\tilde\phi^i(x)=x(1-x)p_{i-2}^{\frac{5}{2}}(2x-1)\chi_{[0,1)}(x),\ i=2,3,\ldots,n$ where
$p_{i}^{\frac{5}{2}}(x)$ is the monic ultraspherical polynomial of degree
$i$ in $x$ then 
$\{\tilde\phi^i\}_{i=2}^n$ form an orthogonal set in $L^2(I)$. Let $r=x\chi_{[0,1)}(x)$, $l=(1-x)\chi_{[0,1)}(x)$,
\begin{equation}
r_n=(I-P_{\{\tilde\phi^k\, |\, k=2,\ldots, n\}})r,\ \text{and}\ l_n=(I-P_{\{\tilde\phi^k\, |\, k=2,\ldots, n\}})l.
\end{equation}
(Note: Here, and following, when $A$ is a finite set of functions, we will write $P_A$ to denote the more cumbersome $P_{\spam A}$.)

Some useful integrals that we will need later are
\begin{equation}\label{orthoone}
\|\tilde\phi^n\|^2=\frac{ (n-2)!(n+2)!}{16(2n-1)!!(2n+1)!!}, 
\end{equation}
\begin{equation}\label{orthotwo}
  \ip{r,\tilde\phi^n}=\frac{(n-2)!}{4(2n-1)!!},\ \ip{l,\tilde\phi^n}=(-1)^n\ip{r,\tilde\phi^n},
\end{equation}
and
\begin{equation}\label{r0l0}
\ip{r_n,l_n}=(-1)^{n+1}\frac{1}{n(n+1)(n+2)}.
\end{equation}
Here we use the notation that $(2n-1)!!=(2n-1)(2n-3)\cdots(1)$. These
formulas may be obtained from \cite{DGHorpoly} (cf., Equation~2.3 with $m=2$; Equation~2.5 with $i=0$, $k=0$;
and Equation~5.3 where $\tilde\phi^n(x)=\frac{1}{4}\phi_n^0(2x-1)\chi_{[0,1)}(x)$, 
$r(x)=\frac{1}{2} r_0^0(2x-1)\chi_{[0,1)}(x)$ and $l(x)=\frac{1}{2} l_0^0(2x-1)\chi_{[0,1)}(x)$).
Since $r_n$ and $l_n$ are not orthogonal we add a function $z^n$
chosen so that when it is projected out from the above two functions they
become orthogonal. To accomplish this set, 
\begin{equation}\label{zn}
z^n=\alpha_n\tilde\phi^{n+1}+\tilde\phi^{n+3},
\end{equation}
where $\alpha_n$ is fixed so that 
\begin{equation}\label{proj}
\|z^n\|^{2}\ip{r_n,(I-P_{\{z^n\}})l_n}=\ip{r_n,l_n}\|z^n\|^2-(-1)^{n+1}\ip{r,z^n}^2=0,\end{equation}
where
\eqref{orthotwo} has been used to obtain the above equation.
The choice of $z^n$ above is for simplicity and  to preserve
the symmetry properties of the basis being constructed.
Substituting \eqref{zn} into \eqref{proj} and using the integrals computed above yields the quadratic equation 
$$
\alpha_n^2+\left(\frac{2(n+1)}{2n+5}\right)\alpha_n+\left(\frac{(n+2)(n+1)(n^2-5n-30)}{(2n+7)((2n+5)^2(2n+3)}\right)=0.
$$  Choosing the positive square root in the quadratic formula we obtain the solution
 \begin{equation}
\alpha_n=-\frac{n+1}{2n+5}+\frac{n+3}{2n+5}\sqrt{\frac{3(n+1)(n+3)}{(2n+7)(2n+3)}}.
\end{equation}
Now we set 
$$
r^n =(I-P_{\{\tilde\phi^2,\ldots,\tilde\phi^n,z^n\}})r=(I-P_{\{z^n\}})r_n $$ and 
$$
l^n=(I-P_{\{\tilde\phi^2,\ldots,\tilde\phi^n, z^n\}})l=(I-P_{\{z^n\}})l_n .
$$
Let $\tilde\Pi^n=\text{span}\{\Pi^n,z^n\}$. The above construction shows,
\begin{lemma}\label{pin}
The functions $\{r^n,l^n,z^n,\tilde\phi^i\ i=2,\ldots,n\}$ form
an orthogonal basis for  $\tilde\Pi^n$. Furthermore $\tilde\Pi^n\subset
\tilde\Pi^{n+3}$.
\end{lemma}
 We now construct a continuous orthogonal basis centered on
 $\mathbf{a}$. Let $\sigma_{a}$ be the affine function taking $a$ to 0 and
$a_{+}$ to 1. For each $a< \sup J$ let
 $\tilde\phi^n_a=\tilde\phi^n\circ\sigma_a$, 
 $z^n_a=z^n\circ\sigma_a$, $r^n_a=r^n\circ\sigma_a$, and $l^n_a=l^n\circ\sigma_a$. For $a<\sup J$ define
\[
\breve{\Omega}^n_{a}=\left\{
\begin{array}
[c]{cc}%
\{l^n_a, z^n_{a},\tilde\phi^2_a,\ldots,\
\tilde\phi^n_a\} & \text{if
}a=\inf J\\
\{z^n_{a},\tilde\phi^2_a,\ldots,
\tilde\phi^n_a\} & \text{if }\inf J<a<a_{+}<\sup J\\
\{r^n_a,z^n_{a},\tilde\phi^2_a,\ldots,\
\tilde\phi^n_a\} & \text{if
}a_{+}=\sup J
\end{array}
\right.  \text{.}%
\]
and, for $\inf J<a<\sup J$, 
$$\bar{\Omega}^n_{a}=\{r^n_{a_{-}}+l^n_{a}\}.$$ 
From the above construction we see that $\breve{\Omega}^n_{a}$ and $\bar{\Omega}^n_{a}$, for $a\in \mathbf{a}$, 
form an orthogonal basis $\Omega^n_{\mathbf{a}}$ centered on the knot sequence $\mathbf{a}$ and that
\begin{equation}\label{nestone}
S_n^0(\mathbf{a})\subset S(\Omega^n_{\mathbf{a}})\subset S^0_{n+3}(\mathbf{a})\subset
S(\Omega^{n+3}_{\mathbf{a}})
\end{equation}
 which implies
\begin{proposition} \label{propMRAPolyRep} Let $\mathbf{a}$ be a knot sequence and, for a fixed   $\kappa\in\{1,2,3\}$, let
  $V^k:=S(\Omega^{3k+\kappa}_{\mathbf{a}}), k=0,1, 2,\ldots$. Then   $V^k\subset V^{k+1}$  and
$\bigcup_{k=0}^{\infty} V^k$ is dense in $L^2(\mathbf{J})$.
\end{proposition}
\begin{proof}
The first two statements follow from equation~\eqref{nestone}. The same
equation also implies that
$\bigcup_{k=0}^{\infty}S_{3k+\kappa}^0(\mathbf{a})\subset\bigcup_{k=0}^{\infty}
V^k$. The third statement now follows by the density of the piecewise 
continuous polynomials with breakpoints on $\mathbf{a}$ in  $L^2(\mathbf{J})$.
\end{proof}

\subsection{Continuous, orthogonal, piecewise quadratic, bases centered on nested knot sequences.
\label{ex1}}

Let $J$ be an interval in $\mathbf{R}$, $\mathbf{a}$ a knot sequence in $J$,
and let $S_{2}^{0}(\mathbf{a})$ denote the spline space consisting of
continuous piecewise quadratic spline functions defined on $J~$with break
points in $\mathbf{a}$. Then $S_{2}^{0}(\mathbf{a})=S(\Phi)$ for a basis
$\Phi$ centered on $\mathbf{a}$. For $a<\sup J$, we define the piecewise
quadratic function $\breve{\phi}_{a}(x)=(x-a)(a_{+}-x)\chi_{\lbrack a,a_{+}%
]}(x)$, where $\chi_{A}$ denotes the characteristic function of a set
$A\subset\mathbf{R}$. Then for $a<\sup J$,%
\[
\breve{V}_{a}=\left\{
\begin{array}
[c]{cc}%
\text{span }\{\breve{\phi}_{a},(\frac{a_{+}-\,\cdot}{a_{+}-a})\chi_{\lbrack
a,a_{+}]}\} & \text{if }a=\inf J\\
\text{span }\{\breve{\phi}_{a}\} & \text{if }\inf J<a<a_{+}<\sup J\\
\text{span }\{\breve{\phi}_{a},(\frac{\cdot\,-a}{a_{+}-a})\chi_{\lbrack
a,a_{+}]}\} & \text{if }a_{+}=\sup J
\end{array}
\right.  \text{.}%
\]
Also, for $\inf J<a<\sup J$,  $\bar{\Phi}_{a}$ can be chosen to consist of the
piecewise linear spline $\bar{\phi}_{a}$ that is 1 at $a$ and 0 at $b$ for
$b\neq a$, $b\in\mathbf{a}$.

We first construct a collection of orthogonal functions supported on the
interval $[0,1)$. These functions will then be used to construct an orthogonal 
continuous piecewise quadratic basis centered on the knot sequence
$\mathbf{a}$. This construction was first given in \cite{DGHsqueeze} (we warn
the reader that there are some typographical errors in the intermediate
computations given in that paper that we now take the opportunity to correct).

Let $r$ and $l$ be as in the previous example and let $q:=4rl$ denote the
\textquotedblleft quadratic bump\textquotedblright\ function of height one on
$[0,1)$. For $0<\theta<1$, let
\[
q_{0}^{\theta}(x):=q(x/\theta),\qquad q_{1}^{\theta}(x):=q\left(
(x-\theta)/(1-\theta)\right)  ,
\]
and
\[
h^{\theta}(x)=r(x/\theta)+l(\left(  x-\theta)/(1-\theta)\right)  .
\]

Now we look for a function $z^{\theta}\in$ span$\{q_{0}^{\theta},q_{1}%
^{\theta},h^{\theta}\}$ so that 
\begin{equation}
\left(  I-P_{\text{span}\{q,z^{\theta}\}}\right)r\perp l
\label{zeqn}%
\end{equation}
and $z^{\theta}\perp q$.

Note that $q$ is in the 3-dimensional space span$\{q_{0}^{\theta}%
,q_{1}^{\theta},h^{\theta}\}$. A basis for the 2-dimensional orthogonal
complement of $q$ in this space is given by (with help from \emph{Mathematica}%
)
\begin{align}
u_{0}  &  =(1-\theta)^{2}(2+3\theta)q_{0}^{\theta}+\theta^{2}(3\theta
-5)q_{1}^{\theta}\nonumber\\
u_{1}  &  =(-2+3\,\left(  -1+\theta\right)  \,\theta^{3})q_{0}^{\theta
}+(-2+3\,{\left(  -1+\theta\right)  }^{3}\,\theta)q_{1}^{\theta}+(\frac{16}%
{5}-12\,{\left(  -1+\theta\right)  }^{2}\,\theta^{2})h^{\theta}.\nonumber
\end{align}
Then $z^{\theta}$ must be of the form $z^{\theta}=c_{0}u_{0}+c_{1}u_{1}.$
Again using \emph{Mathematica} we find that condition \eqref{zeqn} is
equivalent to the following quadratic equation in the variable $c=c_{0}%
/c_{1}$
\begin{align}
0=  &  5\,\left(  4-5\,\left(  1-\theta\right)  ^{2}\,\theta^{2}\,\left(
15+\left(  1-\theta\right)  \,\theta\right)  \right) \label{IMRAeqn}\\
&  \qquad-20\,\left(  2+\theta\,\left(  9+13\,\theta\,\left(  -3+2\,\theta
\right)  \right)  \right)  \,c+4\,\left(  1+45\,\left(  1-\theta\right)
\,\theta\right)  \,c^{2}.\nonumber
\end{align}
The discriminant of this equation is $80\,{\left(  4-15\,{\left(
1-\theta\right)  }^{2}\,\theta^{2}\right)  }^{2}$ and thus is strictly
positive for $\theta\in(0,1)$ giving the two solutions
\[
c=\frac{20(2+\theta(9+13\theta(2\theta-3)))\pm4\sqrt{5}(4-15(1-\theta
)^{2}\theta^{2})}{8(1+45(1-\theta)\theta)}.
\]
Hence, for $0<\theta<1$, there is some $z^{\theta}$ in the span of $u_{0}$ and
$u_{1}$ such that \eqref{zeqn} holds. 
Let $r^{\theta}=\left(  I-P_{\text{span}\{q,z^{\theta}\}}\right)r$ and 
$l^{\theta}=\left(  I-P_{\text{span}\{q,z^{\theta}\}}\right)l$. Then
$\{r^{\theta},l^{\theta},q,z^{\theta}\}$ is an orthogonal
system spanning a four dimensional space of piecewise quadratic functions
continuous on $[0,1)$.

We are now prepared to construct a continuous orthogonal basis  centered on $\mathbf{a}$ using the functions $\{r^{\theta},l^{\theta},q,z^{\theta}\}$.   For $a\in\mathbf{a}$ and $a<\sup J$, let 
   $\theta_a$ be in $(0,1)$ and let 
$\boldsymbol{\theta}$ denote the sequence $(\theta_{a})_{a<\sup J}$. For each $a<\sup J$, let $z_{a}:=z^{\theta_{a}}\circ
\sigma_{a}$ where $\sigma_{a}$ is the affine function taking $a$ to 0 and
$a_{+}$ to 1. Let $$Z=Z_{\mathbf{a},\boldsymbol{\theta}}:=\text {clos}_{L^2(J)} \text{span } \big( \bigcup_{a<\sup J}\, z_{a} \big)%
,$$ then $V=S_{2}^{0}(a)+Z$ satisfies the hypotheses of
 Theorem~\ref{LemmaOrthCond} and so $V=S(\Omega)$ for some orthogonal basis
$\Omega=\Omega_{\mathbf{a},\boldsymbol{\theta}}$ centered on $\mathbf{a}$.
Specifically, $\Omega$ can be chosen as follows: For $a<\sup J$,%
\[
\breve{\Omega}_{a}=\left\{
\begin{array}
[c]{cc}%
\{l^{\theta_{a}}\circ\sigma_{a},q\circ\sigma_{a},z_{a}\} & \text{if
}a=\inf J\\
\{q\circ\sigma_{a},z_{a}\} & \text{if }\inf J<a<a_{+}<\sup J\\
\{r^{\theta_{a}}\circ\sigma_{a},q\circ\sigma_{a},z_{a}\} & \text{if
}a_{+}=\sup J
\end{array}
\right.  \text{.}%
\]
and, for $\inf J<a<\sup J$, $\bar{\Omega}_{a}=\{r^{\theta_{a_{-}}%
}\circ\sigma_{a_{-}}+l^{\theta_{a}}\circ\sigma_{a}\}$. Let
$\mathbf{b}=\mathbf{b}(\mathbf{a},\boldsymbol{\theta})$ denote the sequence with elements
$b_{a}:=(1-\theta_{a})a+\theta_{a}a_{+}$ between $a$ and $a_{+}$ for $a<\sup
J$. Then
\begin{equation}
S_{2}^{0}(\mathbf{a})\subset S(\Omega_{\mathbf{a},\boldsymbol{\theta}})\subset
S_{2}^{0}(\mathbf{a}\cup\mathbf{b}).\label{S02}%
\end{equation}

\begin{lemma}
\label{c0quadlemma} Let $\mathbf{a}^{0}\subset\mathbf{a}^{1}$ be knot
sequences in $J$, let $\boldsymbol{\theta}^{0}$(indexed by $\mathbf{a}^{0}$) and
$\boldsymbol{\theta}^{1}$ (indexed by $\mathbf{a}^{1}$) be parameter sequences
taking values in $(0,1)$. 
For  $a\in\mathbf{a}^{0}$ let $a_+$ denote the successor to $a$ in $\mathbf{a}^0$.   Let $\mathbf{b^{0}}:=\mathbf{b}(\mathbf{a^{0}%
},\boldsymbol{\theta}^{0})$, $\mathbf{b^{1}}:=\mathbf{b}(\mathbf{a^{1}},\boldsymbol{\theta}^{1})$ 
be such that $b_{a}^{0}%
\in\mathbf{a}^{1}$ whenever 
  $(a,a_{+})$ contains some point in $\mathbf{a}^{1}$, and 
$b_{a}^{0}\in\mathbf{b}^{1}$ whenever  $(a,a_{+})$ contains no point of $\mathbf{a}^{1}$.

Then $$S(\Omega_{\mathbf{a^{0}},\boldsymbol{\theta
^{0}}})\subset S(\Omega_{\mathbf{a^{1}},\boldsymbol{\theta^{1}}}).$$
\end{lemma}

\begin{proof}
Let $Z^{\epsilon}=Z_{\mathbf{a}^{\epsilon},\theta^{\epsilon}}$ for
$\epsilon=0,1$. Then $S(\Omega_{\mathbf{a}^{\epsilon},\theta^{\epsilon}%
})=S_{0}^{2}(\mathbf{a}^{\epsilon})+Z^{\epsilon}$. Since $S_{0}^{2}%
(\mathbf{a}^{0})\subset S_{0}^{2}(\mathbf{a}^{1})$ it suffices to show that
$Z^{0}\subset S_{0}^{2}(\mathbf{a}^{1})+Z^{1}$. Let $a\in\mathbf{a}^{0}$ with
$a<\sup J$. If $(a,a_{+})$  contains no point
of $\mathbf{a}^{1}$ then $a_+$ 
 is also the successor of $a$ in $\mathbf{a}^{1}$. Otherwise, $b_{a}^{0}%
\in\mathbf{a}^{1}$ in which case $z_{a}^{0}\in S(\mathbf{a}^{1})$ and so, in
either case, we have $z_{a}^{0}\in S_{0}^{2}(\mathbf{a}^{1})+Z^{1}$ which
completes the proof.
\end{proof}

\begin{proposition} \label{c0quadMRA} Suppose $(\mathbf{a}^k)_{k\in \Z}$ is a sequence of nested knot sequences in an interval $J$ and $(\boldsymbol{\theta}^k)_{k\in \Z}$ is a sequence of parameter sequences taking values in (0,1)   such that for each $k\in \Z$, $(\mathbf{a}^k,\boldsymbol{\theta}^k)$ and $(\mathbf{a}^{k+1},\boldsymbol{\theta}^{k+1})$ satisfy the hypotheses of Lemma~\ref{c0quadlemma}.   Let
$V^k:=S(\Omega_{\mathbf{a}^k,\boldsymbol{\theta}}^k)$.  The following statements hold.
\begin{enumerate}
\item  
$ \displaystyle V^k\subset V^{k+1}$ for $(k\in \Z).$
\item If $\bigcup_k \mathbf{a}^k$ is dense in $J$, then $\bigcup_k V^k$ is dense in $L^2(J)$. 
\item  Let $\boldsymbol{\alpha}:=\bigcap_k \mathbf{a}^k\cup \mathbf{b}({a}^k,\boldsymbol{\theta}^k)$.  Then $\bigcap_k V^k\subset S_2^0(\boldsymbol{\alpha})$. 
\end{enumerate}
 \end{proposition}
 
\begin{rem}  In the above, $\boldsymbol{\alpha}$ may not be a knot sequence in $J$, i.e., it may be empty, a singleton, or it may be that  $\sup \boldsymbol{\alpha}< \sup J$ or $\inf \boldsymbol{\alpha}> \inf J$.  In these cases we define 
 $S_2^0(\boldsymbol{\alpha})$ to be the collection of continuous functions on $J$ whose restriction to any interval in $J\setminus \boldsymbol{\alpha}$ is in $\Pi^d$.  For example $S_2^0(\emptyset)=\Pi^d$. We also mention that if $\bigcup_k \mathbf{a}^k$ is dense in $J$, then $\mathbf{b}({a}^k,\boldsymbol{\theta}^k)\subset \mathbf{a}^\ell$ for some $\ell$ and so $\boldsymbol{\alpha}:=\bigcap_k \mathbf{a}^k$ in this case. 
 \end{rem}

\begin{proof}
Part (1) follows directly from Lemma~\ref{c0quadlemma}.  
Equation \eqref{S02} gives
$$S_2^0(\mathbf{a}^k)\subset V_k \subset S_2^0(\mathbf{a}^k\cup \mathbf{b}({a}^k,\boldsymbol{\theta}^k)) $$
which shows (a) $\bigcup_k S_2^0(\mathbf{a}^k)\subset\bigcup_k V^k  $ and (b) $\bigcap_k V^k \subset \bigcap_k S_2^0(\mathbf{a}^k\cup \mathbf{b}({a}^k,\boldsymbol{\theta}^k))= S_2^0(\boldsymbol{\alpha})$
which implies parts (2) and (3). 
 \end{proof}

\begin{rem}  With $V^k$ as in Proposition~\ref{c0quadMRA}, let $W^k$ denote the orthogonal complement of $V^k$ in $V^{k+1}$ and we have the usual decompostion
$$L^2(J)= (\bigcap_k V^k)\oplus\bigoplus_k W^k.$$
\end{rem}

\section{Wavelets}

\label{Wavelets}

In this section we consider nested spaces $V^0\subset V^1$   generated by orthogonal bases $\Phi^0$ and $\Phi^1$, respectively, that are centered on the same  knot sequence $\mathbf{a}$.   For example, if $(\mathbf{a}^0,\boldsymbol{\theta}^0)$ and $(\mathbf{a}^1,\boldsymbol{\theta}^1)$ satisfy the conditions of Lemma~\ref{c0quadlemma}, then $V^0=S(\Omega_{\mathbf{a}^0,\boldsymbol{\theta}^0})$
and $V^1=S(\Omega_{\mathbf{a}^1,\boldsymbol{\theta}^1})$ are such spaces centered on the {\em same} knot sequence $\mathbf{a}^0$.   A second example, for $n\ge 2$, is given by  $V^0=S(\Omega_{\mathbf{a}}^n)$ and $V^1=S(\Omega_{\mathbf{a}}^{n+3})$  where $\Omega_{\mathbf{a}}^n$ is 
as in Section~\ref{ex1.2}. 

Our main result is that 
the \textbf{wavelet space} $W:=V^1\ominus V^0$ has an 
orthogonal basis $\Psi$ centered on $\mathbf{a}$.

\begin{theorem} \label{ThmWavelets}
Suppose $\Phi^{0}$ and $\Phi^{1}$ are orthogonal bases centered on a common knot sequence $\mathbf{a}$ 
such that $V^0\subset V^1$ where $V^0=S(\Phi^0)$ and $V^1=S(\Phi^1)$. Then there exists an
orthogonal basis $\Psi$ centered on $\mathbf{a}$ such that $W=S(\Psi)$ satisfies $V^1=V^0\oplus W$.
\end{theorem}

\begin{proof}
 The spaces $S(\Phi^0)$ and $S(\Phi^1)$ satisfy the conditions in Theorems~\ref{LemmaSpace2Basis} and~\ref{LemmaOrthCond}.  Let $V^0 := S(\Phi^0)$ and $V^1 := S(\Phi^1)$.  For each $a\in\mathbf{a}$, define 
 $$\bar V^\epsilon_a:=\left(  I-P_{\breve{V}^\epsilon_{a_{-}}\oplus\breve{V}^\epsilon_{a}}\right)V^\epsilon_a, \qquad \epsilon=0,1$$    as in  Theorem~\ref{LemmaOrthCond}.  It follows from (\ref{EqnOrthCondAlt})  that  $\bar V^\epsilon_a \perp \bar V^\epsilon_{a_+}$ and  $V^\epsilon_a = \breve V^\epsilon_{a_{-}} \oplus \bar V^\epsilon_a \oplus \breve V^\epsilon_a$   for all $a\in\mathbf{a}$ and $\epsilon=0,1$.   Letting $\bar V^\epsilon:=S(\bar \Phi^\epsilon)$ and $\breve V^\epsilon:=S(\breve \Phi^\epsilon)$, we have $V^\epsilon=\bar V^\epsilon\oplus\breve V^\epsilon$  for $\epsilon=0,1$. (Note: Also $\bar V^0_a \perp \bar V^1_{a_+}$, which is a simple consequence of the fact that $\Phi^{0}$ and $\Phi^{1}$ are orthogonal bases centered on $\mathbf{a}$ and $V^0\subset V^1$.)

Let  $W:=V^1\ominus V^0$.  First we describe the spaces $\breve W_a:=W_{[a,a_+]}=\breve{V}^1_a\cap (V^0)^\perp$ and
$W_a:=W_{[a_-,a_+]}$ for $a\in \mathbf{a}$.  Then we verify that $W$ satisfies the 
hypotheses of Theorems~\ref{LemmaSpace2Basis} and~\ref{LemmaOrthCond}. 

For   $a\in\mathbf{a}$, observe that $\bar{V} _a^{0}\subset V _a^{0}\subset V _a^{1}%
=\breve{V} _{a_{-}}^{1}\oplus\bar{V} _a^{1}\oplus\breve{V} _a^{1}$, and thus 
\begin{equation}\label{wavePV1b}
\bar{V} _a^{0}\subset P_{\breve{V} _{a_{-}}^{1}}\bar{V} _a 
^{0}\oplus P_{\bar{V} _a^{1}}\bar{V} _a^{0}\oplus P_{\breve{V} 
 _a^{1}}\bar{V} _a^{0}\subset A_a^-\oplus \bar{V} _a^{1}\oplus A_a^+
 \end{equation}
 where, for $a\in \mathbf{a}$, 
\begin{equation}\label{Adef}
A_a^+:=P_{\breve{V} _{a}^{1}}\bar{V} _{a}^{0} \text{ and }A_a^-:=P_{\breve{V} _{a_-}^{1}}\bar{V} _{a}^{0}.
\end{equation}  

\begin{lemma}\label{LemPerp}
For $a\in\mathbf{a}$, the spaces $\breve{V} _a^0$, $A_a^+$, and $A_{a_+}^-$ are mutually orthogonal subspaces of $\breve{V} _a^1$ and 
$P_{\breve{V}^1_a}V^0=\breve{V} _a^{0}\oplus A_a^+\oplus A_{a_+}^-.$
\end{lemma}
\begin{proof} Since $\bar V^0_a,\bar V^0_{a_+}  \perp  \breve V^0_a $ and $ \breve V^0_a \subset  \breve V^1_a $, 
we have   $\bar V^0_a,\bar V^0_{a_+} \perp  P_{\breve
{V} _a^1}\breve V _a^{0}$. Hence, $ A_a^+:=P_{\breve
{V} _a^{1}}\bar V^0_a$ and  $ A_{a_+}^-:=P_{\breve
{V} _a^{1}}\bar V^0_{a_+}$ are both perpendicular to $\breve V^0_a$ (note that if $U_0$ and $U_1$ are subspaces in a Hilbert space   and $P$ is an orthogonal projection  such that $U_0\perp   PU_1$ then $PU_0\perp   U_1$).   Let $f\in \bar V^0_a$. By equation~\eqref{wavePV1b} we have $f=P_{\breve V _{a_{-}}^1}f+ P_{\bar V_a^1}f+ P_{\breve V_a^1}f$.  Similarly, if $g\in \bar V^0_{a_+}$, 
  then  $g=P_{\breve V_a^1}g+ P_{\bar V^1_{a_+}}g+ P_{\breve V^1_{a_+}}g$.  Hence, we have $0=\ip{f,g}=\ip{P_{\breve V_a^1}f,P_{\breve V_a^1}g}$ by the orthogonality assumption of $V^1$.  Hence, $A_a^+\perp A_{a_+}^-$.    The final equation follows from 
  $P_{\breve{V}^1_a}V^0=P_{\breve{V}^1_a}(\breve{V} _a^{0}\oplus \bar{V}^1_{ a_-}\oplus \bar{V}^1_{ a}).$
  \end{proof}
  
 From Lemma~\ref{LemPerp} we obtain   
\begin{equation}\label{breveWa1}
\breve{W}_a=\breve{V}^1_a\cap (V^0)^\perp=\breve{V}^1_a\cap (P_{\breve{V}^1_a}V^0)^\perp=\breve{V}^1_a\cap (\breve{V} _a^{0}\oplus A_a^+\oplus A_{a_+}^-)^\perp,
\end{equation}
and so
\begin{equation}\label{breveV1a}
\breve{V}^1_a=\breve{V} _a^{0}\oplus A_a^+\oplus A_{a_+}^-\oplus \breve{W}_a.
\end{equation}

Then, again using  Lemma~\ref{LemPerp}, we have 
\begin{equation}\label{breveV1aV0perp}
\begin{split}
(I-P_{V^0})\breve{V}_a^1&\subset (I-P_{V^0})A_a^++(I-P_{V^0})A_{a_+}^-+\breve{W}_a \\
&= (I-P_{\bar{V}_a^0})A_a^++(I-P_{\bar{V}_{a_+}^0})A_{a_+}^-+\breve{W}_a \\
&\subset W_a+W_{a_+} 
\end{split}
\end{equation}
where we observe that  $(I-P_{\bar{V}_a^0})A_a^+\subset (\bar{V}_a^0+A_a^+)\cap V_0^\perp\subset W_a$
and similarly $(I-P_{\bar{V}_{a_+}^0})A_{a_+}^-\subset W_{a_+}$.
Furthermore, 
\begin{equation}\label{barV1aV0perp}
(I-P_{V^0})\bar{V}_a^1=(I-P_{V_a^0})\bar{V}_a^1\subset W_a, 
\end{equation}
and, hence, 
\begin{equation}\label{V1aV0perp}
(I-P_{V^0})V_a^1 \subset W_{a_-}+ W_{a}+W_{a_+}.
\end{equation}
Letting $\Sigma_{a\in \mathbf{a}}^*V_a$ denote $\text {clos}_{L^2(J)}\spam\left(\bigcup_{a\in \mathbf{a}}V_a\right))$ for a space $V$ centered on $\mathbf{a}$, we have
\begin{equation}
W=(I-P_{V^0})\Sigma^*_{a\in \mathbf{a}}V_a^1
=\Sigma^*_{a\in \mathbf{a}}(I-P_{V^0})V_a^1
=\Sigma^*_{a\in \mathbf{a}}W_a,
\end{equation}
 that is, $W$ satisfies $(\hat b)$ of Theorem~\ref{LemmaSpace2Basis}.

If $f\in W$ vanishes on $[a,a_{+}]$, then  $f=f_1+f_2$ where $f_1\in V^1_{(-\infty,a ]}$
and $f_2\in V^1_{[a_+,\infty )}$ (property $(\hat c)$ of Theorem~\ref{LemmaSpace2Basis} applied to
$V^1$).  Since $f\perp V^0$ and using the support properties of $f_1$ and $f_2$,
we have $$0=P_{V^0}f=P_{V^0}f_1+P_{V^0}f_2= P_{\bar V^0_a}f_1+P_{\bar V^0_{a_+}}f_2.$$
Since $\bar V^0_a\perp \bar V^0_{a_+}$ we must have $P_{\bar V^0_a}f_1=P_{\bar V^0_{a_+}}f_2=0$ and hence $f_1\in W_{(-\infty,a]}$ and $f_2\in W_{[a,\infty)}$ which verifies that $W$ satisfies property $(\hat c)$ of Theorem~\ref{LemmaSpace2Basis}.  Since $W_a\subset V^1_a$ is finite dimensional, then  Theorem~\ref{LemmaSpace2Basis} implies $W=S(\Psi)$ for some basis $\Psi$ centered on $\mathbf{a}$.

Observe that 
\begin{equation}
\begin{split}
W_a&\subset V^1_a \cap (\breve{V}^0_a\oplus \bar{V}^0_{a_+})^\perp\subset 
\breve{V}^1_{a_-}\oplus \bar{V}^1_a\oplus (\breve{V}^1_a \cap (\breve{V}^0_a\oplus \bar{V}^0_{a_+}))^\perp\\
&= \breve{V}^1_{a_-}\oplus \bar{V}^1_a\oplus A_a^+\oplus \breve{W}_a,
\end{split}
\end{equation}
and in the same way we have
\begin{equation}
W_{a_+}\subset  A_{a_+}^-\oplus\breve{W}_a\oplus\breve{V}^1_{a_+}\oplus \bar{V}^1_{a_+}.
\end{equation}
It then follows that $(I-P_{\breve{W}_a})W_a\perp W_{a_+}$ and so by  Theorem~\ref{LemmaOrthCond} we conclude that $W=S(\Psi)$ for some orthogonal basis centered on $\mathbf{a}$. 

We recall that in the case that  $J$ is not an open interval and  $a\in\mathbf{a}$ is an endpoint of $J$, say $a=\inf J$,  that $a_-=-\infty$ and that the spaces
$\breve V^\epsilon_{a_-}$, $\bar V^\epsilon_a$ are   trivial  for $\epsilon=0,1$.  It then follows that $\bar{W}_a=\{0\}$ whenever $a$ is an endpoint of $J$. 
\end{proof}
The spaces $A_a^\pm$ defined in \eqref{Adef} play central roles in the above proof. 
In the following corollary, we express $\breve{W}_a$ and $\bar{W}_a$ in terms of these spaces.

\begin{corollary}\label{WaveCor}
Suppose $V^0$, $V^1$, and $W$ are as in Theorem~\ref{ThmWavelets} and let $A_a^\pm$ be given by \eqref{Adef}.  Let $\bar W_a:=W_a\ominus({\breve{W}_{a_-}\oplus \breve{W}_a})$.  Then
\begin{equation}\label{breveWa}
\breve{W}_a=\breve{V}^1_a\cap (\breve{V} _a^{0}\oplus A_a^+\oplus A_{a_+}^-)^\perp,
\end{equation}
and 
\begin{equation}\label{barWa}
\bar{W}_a=\left(A_a^-\oplus \bar V_a^1 \oplus A_a^+\right)\cap \left( \bar V_a^0\right)^\perp.
\end{equation}
\end{corollary}

\begin{proof}
The formula \eqref{breveWa} is given in \eqref{breveWa1}.  Proceeding as in the previous proof, 
\begin{equation}\label{Wa}
\begin{split}
W_a&=V_a^1\cap \left( V_{a_-}^0\oplus \bar V_a^0\oplus V_{a_+}^0\right)^\perp
\\ &= \left(A_a^-\oplus \bar V_a^1 \oplus A_a^+\oplus\breve W_{a_-}\oplus \breve W_a\right)\cap \left(\bar V_a^0\right)^\perp\\
&= \left(\left(A_a^-\oplus \bar V_a^1 \oplus A_a^+\right)\cap \left( \bar V_a^0\right)^\perp\right)\oplus\breve W_{a_-}\oplus \breve W_a\\
&=\bar{W}_a\oplus\breve W_{a_-}\oplus\breve W_a.
\end{split}
\end{equation}
which establishes \eqref{barWa}.
\end{proof}

\subsection{Decomposing the wavelet spaces}\label{dws}
\newcommand{\dm}{\text{\rm dim }}

In order to aid in the construction of the wavelets, we discuss a parallel development based on the wavelet construction given in 
\cite{DGHorpoly}.  This construction decomposes $\bar W_a$ into two orthogonal subspaces that illuminate and simplify the construction of the bar wavelets.  We also give more explicit characterizations of the dimensions of the subspaces discussed above. 
 
Let $k^\epsilon_a=\dm V^\epsilon_a$, $\bar k^\epsilon_a=\dm\bar V^\epsilon_a$, $\breve k^\epsilon_a=\dm\breve V^\epsilon_a$, $Y_a=\bar V^0_a\cap\bar
V^1_a$, $m_a=\dm Y_a$,  $Y^-_a=\bar V^0_a\chi_{[a_-,a]}\cap\bar
V^1_a\chi_{[a_-,a]}$, $\dm  Y^-_a=m^-_a$, $Y^+_a=\bar V^0_a\chi_{[a,a_+]}\cap\bar
V^1_a\chi_{[a,a_+]}$, $\dm  Y^+_a=m^+_a$, $X_a^+ = (\chi_{[a,a_+]}\bar V_a^0) \ominus Y_a^+$, and $X_a^-:= (\chi_{[a_-,a]}\bar V_a^0) \ominus Y_a^-$. 

We begin with a sort of algorithm for constructing the ``short" wavelet space, $\breve W_a$, which will show the use of the spaces, $Y_a^+$ and $Y_a^-$. Recall that $A_a^- = P_{\breve V^1_{a_-}}\bar V^0_a$ and $A_a^+ = P_{\breve V^1_a}\bar V^0_a$.

\begin{lemma}\label{lemma10}

\begin{enumerate}

\item $\dim \breve W_a = (\breve k_a^1 -\breve k_a^0)- \dim X_a^+ - \dim X_{a_+}^-$.
\item $A_a^+= P_{\breve V^1_a} X_a^+$, and $A_{a_+}^-= P_{\breve V^1_a} X_{a_+}^-$.
\item $\dim A_a^+ = \dim X_a^+ = \bar k_a^0 - m_a^+$, and $\dim A_{a_+}^ - =  \dim X_{a_+}^- = \bar k_{a_+}^0 - m_{a_+}^-$.
\end{enumerate}

\end{lemma}

\begin{proof} From the definition of $W$ it follows that $\breve W_a \subset \breve V_a^1 \ominus \breve V_a^0$. But $\breve W_a$ is also orthogonal to $\chi_{[a,a_+]}\bar V_a^0 = X_a^+ \oplus Y_a^+$, and $\chi_{[a,a_+]}\bar V_{a_+}^0= X_{a_+}^- \oplus Y_{a_+}^-$. Since $Y_a^+$ and $Y_{a_+}^-$ are already orthogonal to $\breve V_a^1 \ominus \breve V_a^0$, then to construct a basis for $\breve W_a$ we only need to choose a basis from $\breve V_a^1 \ominus \breve V_a^0$ that is simultaneously orthogonal to $X_a^+$ and $X_{a_+}^-$. Further, since $\breve V_a^1 \ominus \breve V_a^0 \subset \breve V_a^1$, we need this basis orthogonal to spaces $P_{\breve V^1_a} X_a^+$ and $P_{\breve V^1_a} X_{a_+}^-$.

Using the fact that  $\bar f^0_a = P_{\breve V^1_{a_-}}\bar f^0_a + P_{\bar V^1_a}\bar f^0_a + P_{\breve V^1_a}\bar f^0_a$, for $\bar f^0_a \in \bar V^0_a$,  it is easy to check that the kernel of $P_{\breve V^1_a}$ acting on $\chi_{[a,a_+]}\bar V^0_a$ is $Y_a^+$. Thus $\dim P_{\breve V^1_a} X_a^+ = \dim X_a^+$. Similarly it can be shown that $\dim P_{\breve V^1_a} X_{a_+}^- = \dim X_{a_+}^-$. It is easy to check that the spaces $P_{\breve V^1_a} X_a^+$ and $P_{\breve V^1_a} X_{a_+}^-$ are subspaces of $\breve V_a^1 \ominus \breve V_a^0$ and are orthogonal to each other. The first equation of the statement of the lemma follows now since the number of additional orthogonality conditions is $\dim X_a^+ + \dim X_{a_+}^-$.

Observe that $A_a^+= P_{\breve V^1_a} \chi_{[a,a_+]}\bar V_a^0 = P_{\breve V^1_a}( X_a^+ \oplus Y_a^+) = P_{\breve V^1_a}X_a^+$. Similarly we can show that $A_{a_+}^-= P_{\breve V^1_a} X_{a_+}^-$. By local linear independence, $\dim \bar V^0_a \chi_{[a,a_+]} = \dim \bar V^0_a = \bar k^0_a$ and so $\dim A_a^+ = \bar k_a^0 - m_a^+$. The formula  for $\dm A_a^-$ follows in the same way.

\end{proof}

Using    Lemma~\ref{lemma10} combined with Corollary~\ref{WaveCor}, we now have

\begin{theorem}\label{waveletdim}
\begin{align*}
\dim \bar W_a &= \bar k^1_a+\bar k^0_a-m^+_a-m^-_a , \\
\dim \breve W_a &= \breve k^1_a - \breve k^0_a -\bar k^0_a-\bar k^0_{a_+}+m_a^+ + m_{a_+}^-,\\
\dim (\breve W_{a_-} \oplus \breve W_a) &= (m_{a_-}^+ - \bar k_{a_-}^0) +( k^1_a-k^0_a-\bar k^1_a -\bar k^0_a +m^-_a + m^+_a )+ (m_{a_+}^- -\bar k_{a_+}^0).
\end{align*}
\end{theorem}

We now introduce spaces that aid in the decomposition of $\bar W_a$, this follows the development given in \cite{DGHorpoly}. Let 

\begin{align*}
T_a&=(I-P_{\bar V^1_a})\bar V^0_a\\
T^-_a&=\{f\in T_a:{\rm supp}f\subset[a_-,a]\}\\
T^+_a&=\{f\in T_a:{\rm supp}f\subset[a,a_+]\}\\
U_a&=T_a\ominus (T^-_a\oplus T^+_a)\\
S_a&=(\chi_{[a,a_+]}-\chi_{[a_-,a]})U_a
\end{align*}

\begin{rem}
In a typical refinement the spaces $T^{\pm}_a$ are trivial. However in section~\ref{cpqtw}
we present an example of a refinement where these spaces are not empty.
\end{rem}

\begin{lemma}\label{suma}
$\displaystyle A_a^- \oplus A_a^+=S_a+T_a$.
\end{lemma}
\begin{proof}
First we remark that 
$$
T_a=(P_{V^1_a}-P_{\bar V^1_a})\bar V^0_a=(P_{\breve V^1_{a}}+P_{\breve V^1_{a_-}})\bar V^0_a\subset  P_{\breve V^1_{a_-}}\bar V^0_a+P_{\breve V^1_a}\bar V^0_a= A_a^- \oplus A_a^+.
$$
Next, define 
$$
\tilde S_a:=(P_{\breve V^1_a}-P_{\breve V^1_{a_-}})\bar V^0_a. 
$$ 
It is clear that
$$A_a^-\oplus A_a^+=\tilde S_a+T_a.$$

Furthermore, 
$$ 
(P_{\breve V^1_a}-P_{\breve V^1_{a_-}})T_a=(P_{\breve V^1_a}-P_{\breve V^1_{a_-}})(P_{\breve V^1_{a}}+P_{\breve V^1_{a_-}})\bar V^0_a=(P_{\breve V^1_a}-P_{\breve V^1_{a_-}})\bar V^0_a=\tilde S_a $$
which shows
 $$\tilde S_a=(P_{\breve V^1_a}-P_{\breve V^1_{a_-}})T_a=(P_{\breve V^1_a}-P_{\breve V^1_{a_-}})(U_a\oplus T_a^-\oplus T_a^+)=S_a +T_a^- + T_a^+;$$ 
 where we used $S_a=(\chi_{[a,a_+]}-\chi_{[a_-,a]})U_a=(P_{\breve V^1_a}-P_{\breve V^1_{a_-}})U_a$ since $U_a\subset A_a^-\oplus A_a^+$. 
 Since $T_a^-$ and $T_a^+$ are contained in $T_a$, we have
 $\tilde S_a+T_a=S_a+T_a$ which completes the proof.   
\end{proof}

\begin{lemma}\label{ydim}
 $\dm T_a=\bar k^0_a-m_a$,  $\dm T^-_a=m^+_a-m_a$,  $\dm T^+_a=m^-_a-m_a$, $\dm U_a=\bar k^0+m_a-m^-_a-m^+_a=\dm S_a$. 
\end{lemma}
\begin{rem}
We take this opportunity to point out that in \cite[page 1041]{DGHorpoly}  the dimension of $T_0$ should be $m_1-m$ and the dimension of $T_1$ should be $m_0-m$.  The analog of these spaces are $T^-_a$ and $T^+_a$.  
\end{rem}
\begin{proof} The linear transformation $I-P_{\bar V^1_a}$ maps $\bar V^0_a$ onto $T_a$. The kernel of this mapping is clearly $Y_a$, and thus the stated dimension of $T_a$ follows.

Once the dimensions of $T^+_a$ and $T^-_a$, as stated, are proved,  the stated dimension of $U_a$ follows easily from its definition. Since $P_{\bar V^1_a}U_a=0$,
\begin{equation}\label{cont}
U_a\subset \breve V^1_{a_-}\oplus\breve V^1_a,
\end{equation} 
which implies that
multiplying functions in $U_a$ by characterisitic functions above does not
destroy membership in $V^1_a$ so the dimension of $S_a$ follows. 

Let $y_a^+ \in Y_a^+$. Then there are functions $\bar f^0_a \in \bar V^0_a$ and $\bar f^1_a \in \bar V^1_a$ so that $y_a^+ = \bar f^0_a \chi_{[a,a_+]}=\bar f^1_a \chi_{[a,a_+]}$. Note that by local linear independence the functions $\bar f^0_a$ and $\bar f^1_a$ are uniquely determined.

It follows that 

$$
(I-P_{\bar V^1_a})\bar f^0_a = P_{\breve V^1_{a_-}}\bar f^0_a +P_{\breve V^1_a}\bar f^0_a = P_{\breve V^1_{a_-}}\bar f^0_a +P_{\breve V^1_a}\bar f^1_a = P_{\breve V^1_{a_-}}\bar f^0_a.
$$
As such, $(I-P_{\bar V^1_a})\bar f^0_a \in T_a^-$.

Conversely, let $t_a^- \in T_a^-$. Choose $\bar f^0_a \in \bar V^0_a$ so that $(I-P_{\bar V^1_a})\bar f^0_a = t_a^-$. Since $(I-P_{\bar V^1_a})\bar f^0_a = P_{\breve V^1_{a_-}}\bar f^0_a +P_{\breve V^1_a}\bar f^0_a$ we get that $P_{\breve V^1_a}\bar f^0_a=0$. By the same argument as given in the proof of Lemma \ref{lemma10} it follows that $\bar f^0_a \chi_{[a,a_+]}\in Y_a^+$.

Define the linear transformation $F: Y^+_a  \rightarrow T^-_a$ by $F(y_a^+)=(I-P_{\bar V^1_a})\bar f^0_a$. The above argument shows that $F$ is onto. It is clear that the kernel of $F$ is $Y_a \chi_{[a,a_+]}$. Since, by local linear independence,  $\dim \chi_{[a_-,a]}Y_a = \dim Y_a = m_a$, it follows that $\dim T_a^- = m_a^+-m_a$. 

The proof that $\dm T^+_a=m^-_a-m_a$ is entirely analogous.
\end{proof}

\begin{rem}
From Lemmas \ref{lemma10}, \ref{suma}, and \ref{ydim},  it follows that $\dim S_a + \dim T_a = \dim (S_a+T_a)$ and thus $S_a$ and $T_a$ are linearly independent.
\end{rem}

We now begin the decomposition of $\bar W_a$.

\begin{lemma}\label{stepone} 
Let $\hat W_a :=(I-P_{\bar V^0_a})\bar V^1_a$. 
Then $\dm \hat W_a=\bar k^1_a-m_a$ and $\hat W_a\subset \bar W_a$.
\end{lemma}

\begin{proof} Since $\bar V^0_a \subset A_a^- \oplus \bar V^1_a \oplus A_a^+$ it follows from Remark \ref{waveletConst} that $\hat W_a \subset \bar W_a$.  The dimension formula for $\hat W_a$ follows from
elementary Hilbert space arguments.
\end{proof}

\begin{lemma}\label{equalspace}
 $\bar V^0_a\oplus\hat W_a=T_a\oplus\bar V^1_a$.
\end{lemma}
\begin{proof}
Since the dimension of the spaces on each side of the above equations are equal
we need only show that one space is contained in the other to prove the result. By definition of $\hat W_a$ it follows that $\bar V^1_a \subset \bar V^0_a\oplus\hat W_a$. This, in turn, implies that $T_a \subset \bar V^0_a\oplus\hat W_a$.
\end{proof}

\begin{lemma} \label{tildew}Let $\tilde W_a :=\bar W_a \ominus \hat W_a$. Then

\begin{enumerate}
\item 
\begin{align*}
\tilde W_a &= (A_a^- \oplus A_a^+) \cap (\bar V_a^0)^{\perp}\\
		&=(I-P_{\bar V_a^0 \oplus \hat W_a})(A_a^- \oplus A_a^+)\\
		&=(I-P_{\bar V_a^0 \oplus \hat W_a})S_a.
\end{align*}
\item $\dim \tilde W_a = \bar k^0_a+m_a-m^-_a-m^+_a$
\end{enumerate}
\end{lemma}

\begin{proof} 

(1) Let $\tilde w_a \subset \tilde W_a$. Obviously, $\tilde w_a \in (\bar V_a^0)^{\perp}$. Choose $f_{a_-,a} \in A_a^-$, $\bar f_a^1 \in \bar V_a^1$, and $f_{a,a} \in A_a^+$ such that $\tilde w_a = f_{a_-,a} + \bar f_a^1 + f_{a,a}$. Now $0=\langle \tilde w_a, \bar f_a^1 - P_{\bar V_a^0}\bar f_a^1\rangle= \langle \tilde w_a, \bar f_a^1 \rangle = \langle \bar f_a^1, \bar f_a^1 \rangle$. Thus $\bar f_a^1=0$ and $\tilde w_a = f_{a_-,a}  + f_{a,a}$. We now have that $\tilde W_a \subset (A_a^- \oplus A_a^+) \cap (\bar V_a^0)^{\perp}$.

Next, choose $f_{a_-,a} \in A_a^-$ and $f_{a,a} \in A_a^+$ such that $f_{a_-,a}  + f_{a,a} \perp \bar V^0_a$. By the definition of $\hat W_a$, it is clear that $f_{a_-,a}  + f_{a,a} \perp \hat W_a$. Thus $(I-P_{\bar V_a^0 \oplus \hat W_a})(f_{a_-,a}  + f_{a,a}) =f_{a_-,a}  + f_{a,a}$. We then get that $(A_a^- \oplus A_a^+) \cap (\bar V_a^0)^{\perp} \subset (I-P_{\bar V_a^0 \oplus \hat W_a})(A_a^- \oplus A_a^+)$.

By the characterization of $\bar W_a$ in the wavelet theorem, it is clear that $(I-P_{\bar V_a^0 \oplus \hat W_a})(A_a^- \oplus A_a^+) \subset \tilde W_a$.

Now we have, by Lemmas  \ref{suma} and \ref{equalspace}, $\tilde W_a   = (I-P_{\bar V_a^0 \oplus \hat W_a})(A_a^- \oplus A_a^+) = (I-P_{\bar V_a^1 \oplus T_a})(S_a+T_a) = (I-P_{\bar V_a^1 \oplus T_a})S_a$.

(2) This follows directly from Theorem \ref{waveletdim} and Lemma \ref{stepone}.

\end{proof}

\begin{rem}
I follows from Lemmas \ref{ydim} and \ref{tildew} that $(\bar V_a^0 \oplus \hat W_a) \cap S_a=0$
\end{rem}

\subsection{Algorithm}\label{waveletConst}
The decomposition given in subsection \ref{dws} provides a
procedure for
constructing an orthonormal wavelet  basis $\Psi$ which we next summarize.
Note that the construction is local and if the scaling functions
are symmetric or antisymmetric on the interval then the wavelets can be constructed which
are symmetric or antisymmetric. 

Fix $a\in \mathbf{a}$ and let $k^\epsilon_a$, $\bar k^\epsilon_a$, $\breve
k^\epsilon_a$, $m_a$, and $m^{\pm}_a$ be as in the first paragraph of
subsection \ref{dws}. We begin with the construction of the
``long'' $\bar W_a$ wavelets. From Lemma~\ref{stepone} 
\begin{equation}\label{wavehat}
\hat W_a =(I-P_{\bar V^0_a})\bar V^1_a
\end{equation}
  has dimension $\bar k^1_a-m_a$, and so we may construct an orthonormal
basis of $\bar k^1_a-m_a$ wavelet functions for this space. If the scaling
functions are symmetric or antisymmetric with respect to the point $a$
then taking linear combinations of the symmetric functions or
antisymmetric functions in $\hat W_a$ allows us to construct wavelets
 that preserve the symmetry.  Next construct $T_a$, then eliminate
the functions in $T_a$ which are supported in $[a_-,a]$ or $[a,a_+]$ to
find $U_a$. Finally apply the difference in characteristic functions to obtain 
$S_a$ as given below Theorem~\ref{waveletdim}. Now by Gram-Schmidt
construct an orthonormal basis of $k^0_a+m_a-m^-_a-m^+_a$ functions for 
\begin{equation}\label{wavetilde}
\tilde W_a=(I-P_{\bar{V}_a^{0}\oplus\hat W_a})S_a
\end{equation}
given by equation~(1) of Lemma~\ref{tildew}. From the construction 
of $S_a$ and the
discussion above we see that if the scaling functions are symmetric or
antisymmetric with respect to $a$ then the wavelets can be constructed 
to preserve the
symmetry. The final step is to construct the wavelets in 
$\breve W_a$. From equation~(\ref{breveWa}) there are 
$\breve k^1_a-\breve k^0_a-\bar k^0_a-\bar
k^0_{a_+}+m^+_a+m^-_{a_+}$ orthonormal functions to be computed in 
\begin{equation}\label{waveshort}
\breve W_a=(I-P_{\breve{V} _a^{0}\oplus A_a^+\oplus A_{a_+}^-})\breve
V^1_a
\end{equation}
to complete the construction of the wavelets.

\subsection{Scaling and Wavelet Matrix Coefficients\label{MatCoeffs}}
In this section we  make use of matrices consisting of inner products. Let $f$ represent a column vector consisting of a finite collection of functions in $L^2(\R)$. If $g$ is another such column we let $\langle f,g \rangle$ denote the matrix so that the $ij$ term is given by $\langle f_i,g_j \rangle$.

Let $\Phi^0$ and $\Phi^1$ be ortho\emph{normal} bases centered on a knot sequence $\mathbf{a}$ as in Theorem~\ref{ThmWavelets}, with associated orthonormal wavelet basis $\Psi$.  For $a' \in \{a_-,a\}$ define
$$
c^{\bar{}~\bar{}}_{aa'}:= \langle  \bar \phi^0_a,  \bar \phi_{a'}^1 \rangle \quad \text{ and } \quad 
c^{\bar{}~\breve{}}_{aa'}:= \langle  \bar \phi^0_a,  \breve \phi_{a'}^1 \rangle.
$$
Similarly, for $a' \in \{a_-,a\}$, we define $c^{\breve{}~\bar{}}_{aa'}$ and $c^{\breve{}~\breve{}}_{aa'}$. (Note: Of course these matrices are only defined when the appropriate collections of functions are non-empty.)
It follows that 
\begin{equation}\label{ceq}
\left[
\begin{matrix}
\breve \Phi_{a_-}^0\\
\bar \Phi_a^0\\
\breve \Phi_a^0
\end{matrix}
\right]=
\left[
\begin{matrix}
c^{\breve{}~\breve{}}_{a_-a_-}&0&0\\
c^{\bar{}~\breve{}}_{aa_-}&c^{\bar{}~\bar{}}_{aa}&c^{\bar{}~\breve{}}_{aa}\\
0&0&c^{\breve{}~\breve{}}_{aa}
\end{matrix}
\right]
\left[
\begin{matrix}
\breve \Phi_{a_-}^1\\
\bar \Phi_a^1\\
\breve \Phi_a^1
\end{matrix}
\right]
\end{equation}
Note that in equation (\ref{ceq}) we can regard the matrix
$$
\left[
\begin{matrix}
c^{\breve{}~\breve{}}_{a_-a_-}&0&0\\
c^{\bar{}~\breve{}}_{aa_-}&c^{\bar{}~\bar{}}_{aa}&c^{\bar{}~\breve{}}_{aa}\\
0&0&c^{\breve{}~\breve{}}_{aa}
\end{matrix}
\right]
$$
as having rows (resp. columns) indexed by 
$$
\left[
\begin{matrix}
\breve \Phi_{a_-}^0\\
\bar \Phi_a^0\\
\breve \Phi_a^0
\end{matrix}
\right]
\left(
~\text{resp.}~
\left[
\begin{matrix}
\breve \Phi_{a_-}^1\\
\bar \Phi_a^1\\
\breve \Phi_a^1
\end{matrix}
\right]
\right).
$$
If one of the row (resp. column) indices is empty then the corresponding row (resp. column) is omitted. For example, if $\bar \Phi_a^0$ and $\bar \Phi_a^1$ are empty then equation (\ref{ceq}) becomes
$$
\left[
\begin{matrix}
\breve \Phi_{a_-}^0\\
\breve \Phi_a^0
\end{matrix}
\right]=
\left[
\begin{matrix}
c^{\breve{}~\breve{}}_{a_-a_-}&0\\
0&c^{\breve{}~\breve{}}_{aa}
\end{matrix}
\right]
\left[
\begin{matrix}
\breve \Phi_{a_-}^1\\
\breve \Phi_a^1
\end{matrix}
\right].
$$
We follow this convention in all the remaining matrix constructions.

For $a' \in \{a_-,a\}$ we can also define matrices $d^{\bar{}~\bar{}}_{aa'}$, $d^{\bar{}~\breve{}}_{aa'}$, $d^{\breve{}~\bar{}}_{aa'}$, and $d^{\breve{}~\breve{}}_{aa'}$ so that

\begin{equation}\label{deq}
\left[
\begin{matrix}
\breve \Psi_{a_-}\\
\bar \Psi_a\\
\breve \Psi_a
\end{matrix}
\right]=
\left[
\begin{matrix}
d^{\breve{}~\breve{}}_{a_-a_-}&0&0\\
d^{\bar{}~\breve{}}_{aa_-}&d^{\bar{}~\bar{}}_{aa}&d^{\bar{}~\breve{}}_{aa}\\
0&0&d^{\breve{}~\breve{}}_{aa}
\end{matrix}
\right]
\left[
\begin{matrix}
\breve \Phi_{a_-}^1\\
\bar \Phi_a^1\\
\breve \Phi_a^1
\end{matrix}
\right]
\end{equation}
For knot $a\in \mathbf{a}$ and $a'\in\{a_-, a\}$  define
$$
c_{aa'}:=
\left[
\begin{matrix} 
c_{aa'}^{\bar{}\ \bar{}} & c_{aa'}^{\bar{}\ \breve{}} \\
c_{aa'}^{\breve{}\ \bar{}} &c_{aa'}^{\breve{}\ \breve{}}
\end{matrix}
\right].
$$
In particular, we have
$$
c_{aa_-}=
\left[
\begin{matrix} 
0 & c_{aa_-}^{\bar{}\ \breve{}} \\
0 &0
\end{matrix}
\right]
~\text{and}~
c_{aa}=
\left[
\begin{matrix} 
c_{aa}^{\bar{}\ \bar{}} & c_{aa}^{\bar{}\ \breve{}} \\
0 &c_{aa}^{\breve{}\ \breve{}}
\end{matrix}
\right].
$$
For $a\in \mathbf{a}$,  equation (\ref{ceq})  implies the following ``scaling equation":
\begin{equation}
\Phi_a^0=\sum_{a'  \in \{a_-,a\}} c_{aa'}\Phi^1_{a'}.
\end{equation}
Similarly we can define
$$
d_{aa'}:=
\left[
\begin{matrix} 
d_{aa'}^{\bar{}\ \bar{}} & d_{aa'}^{\bar{}\ \breve{}} \\
d_{aa'}^{\breve{}\ \bar{}} &d_{aa'}^{\breve{}\ \breve{}}
\end{matrix}
\right].
$$
For $a\in \mathbf{a}$,  equation (\ref{deq})  implies 
\begin{equation}
\Psi_a=\sum_{a'  \in \{a_-,a\}} d_{aa'}\Phi^1_{a'}.
\end{equation}

We next describe the construction of $\Psi$ in terms of the matrix coefficients $d_{ aa'}$. 
When they are defined, the matrices $c_{ aa}^{\breve{}\ \breve{}}$ are full rank   with orthonormal rows as is the block matrix
$$\left[c_{ aa_-}^{\bar{} \ \breve{}}\  c_{ aa}^{\bar{} \ \bar{}}\ c_{ aa}^{\bar{} \ \breve{}}\right].$$  However, the  individual blocks may not be full rank (although this is the generic case).   For $a'=a,a_-$,  if $c_{aa'}^{\bar{}\ \breve{}} \ne 0$ let $b_{ aa'}^{\bar{} \ \breve{}}$ 
be the matrix with orthonormal rows whose row span is the same as $c_{ aa'}^{\bar{} \ \breve{}}$
and let $e_{ aa'}^{\bar{} \ \breve{}}$ be the matrix such that
$$
c_{ aa'}^{\bar{} \ \breve{}}=e_{ aa'}^{\bar{} \ \breve{}}b_{ aa'}^{\bar{} \ \breve{}}.
$$
We observe that 
$$
\alpha_a^+:= b_{ aa}^{\bar{} \ \breve{}}\breve \Phi_a^1
$$
is an orthonormal basis for $A_a^+$, and 
$$
\alpha_a^-:= b_{ aa_-}^{\bar{} \ \breve{}}\breve \Phi_{a_-}^1
$$
is an orthonormal basis for $A_a^-$. (Note: If $b_{ aa}^{\bar{} \ \breve{}}$ (resp. $b_{ aa_-}^{\bar{} \ \breve{}}$) is $0$ then $\alpha_a^+$ (resp. $\alpha_a^-$) is empty.)
From Corollary \ref{WaveCor} we have
$$
\breve{W}_a:=\breve{V} _{a}^{1}\ominus \left(\breve{V} _a^{0}\oplus A_a^+\oplus A_{a_+}^-\right).  $$
If $\breve{W}_a \ne \{0\}$ it follows that $d_{ aa}^{\breve{}\ \breve{}}$ may be chosen so that  
$$
\left[\begin{matrix} 
c_{ aa}^{\breve{}\ \breve{}}\\
b_{ aa}^{\bar{}\ \breve{}}\\
b_{a_+ a}^{\bar{}\ \breve{}}\\
d_{ aa}^{\breve{} \ \breve{}}
\end{matrix}\right]
$$ 
is a square orthogonal matrix which means that $\breve \Psi_a = d_{ aa}^{\breve{}\ \breve{}}\breve \Phi_a^1$.   If $\breve{W}_a = 0$, then $\breve \Psi_a$ is empty.
From equation (\ref{ceq}) we see that, when defined,
\begin{align*}
\bar \Phi_a^0& =
\left[
\begin{matrix}
c^{\bar{}~\breve{}}_{aa_-}&c^{\bar{}~\bar{}}_{aa}&c^{\bar{}~\breve{}}_{aa}
\end{matrix}
\right]
\left[
\begin{matrix}
\breve \Phi_{a_-}^1\\
\bar \Phi_a^1\\
\breve \Phi_a^1
\end{matrix}
\right]\\
	&=
\left[
\begin{matrix}
e^{\bar{}~\breve{}}_{aa_-}b^{\bar{}~\breve{}}_{aa_-}&c^{\bar{}~\bar{}}_{aa}&e^{\bar{}~\breve{}}_{aa}b^{\bar{}~\breve{}}_{aa}
\end{matrix}
\right]
\left[
\begin{matrix}
\breve \Phi_{a_-}^1\\
\bar \Phi_a^1\\
\breve \Phi_a^1
\end{matrix}
\right]\\
	&=
\left[
\begin{matrix}
e^{\bar{}~\breve{}}_{aa_-}&c^{\bar{}~\bar{}}_{aa}&e^{\bar{}~\breve{}}_{aa}
\end{matrix}
\right]
\left[
\begin{matrix}
\alpha_a^-\\
\bar \Phi_a^1\\
\alpha_a^+
\end{matrix}
\right].
\end{align*}
Since 
$$
\left[
\begin{matrix}
\alpha_a^-\\
\bar \Phi_a^1\\
\alpha_a^+
\end{matrix}
\right]
$$
is an orthonormal basis for $A_a^-\oplus \bar V_a^1 \oplus A_a^+$, the matrix
$$
\left[
\begin{matrix}
e^{\bar{}~\breve{}}_{aa_-}&c^{\bar{}~\bar{}}_{aa}&e^{\bar{}~\breve{}}_{aa}
\end{matrix}
\right]
$$
has orthonormal rows.

Again, from Corollary \ref{WaveCor} we have 
 $$  \bar{W}_a:=\left(A_a^-\oplus \bar V_a^1 \oplus A_a^+\right)\ominus \bar V_a^0. \\
$$ 
If $\bar{W}_a \ne \{0\}$ it follows that  the matrices $d_{ aa}^{\bar{}\ \bar{}}$ and $d_{ aa'}^{\bar{}\ \breve{}}$ may be found by completing the matrix 
$\left[e_{ aa_-}^{\bar{}\ \breve{}}  c_{aa}^{\bar{}\ \bar{}} e_{ aa}^{\bar{}\ \breve{}} \right]$ to an orthogonal square matrix, i.e., by determining matrices
$f_{ aa_-}^{\bar{}\ \breve{}}$,  $d_{aa}^{\bar{}\ \bar{}}$, and  $f_{ aa}^{\bar{}\ \breve{}}$ such that 
$$
\left[\begin{matrix} 
e_{ aa_-}^{\bar{}\ \breve{}}  c_{aa}^{\bar{}\ \bar{}} e_{ aa}^{\bar{}\ \breve{}}\\
f_{ aa_-}^{\bar{}\ \breve{}}  d_{aa}^{\bar{}\ \bar{}} f_{ aa}^{\bar{}\ \breve{}}\\
\end{matrix}\right]
$$ 
is an orthogonal matrix. Thus we have 

$$
\left[
\begin{matrix}
\bar \Phi_a^0\\
\bar \Psi_a
\end{matrix}
\right]=
\left[
\begin{matrix} 
e_{ aa_-}^{\bar{}\ \breve{}}  c_{aa}^{\bar{}\ \bar{}} e_{ aa}^{\bar{}\ \breve{}}\\
f_{ aa_-}^{\bar{}\ \breve{}}  d_{aa}^{\bar{}\ \bar{}} f_{ aa}^{\bar{}\ \breve{}}\\
\end{matrix}
\right]
\left[
\begin{matrix}
\alpha_a^-\\
\bar \Phi_a^1\\
\alpha_a^+
\end{matrix}
\right]=
\left[
\begin{matrix} 
c_{ aa_-}^{\bar{}\ \breve{}}  c_{aa}^{\bar{}\ \bar{}} c_{ aa}^{\bar{}\ \breve{}}\\
d_{ aa_-}^{\bar{}\ \breve{}}  d_{aa}^{\bar{}\ \bar{}} d_{ aa}^{\bar{}\ \breve{}}\\
\end{matrix}
\right]
\left[
\begin{matrix}
\bar \Phi_{a_-}^1\\
\bar \Phi_a^1\\
\bar \Phi_a^1
\end{matrix}
\right],
$$
where $d_{ aa'}^{\bar{}\ \breve{}}=f_{ aa'}^{\bar{}\ \breve{}}b_{ aa'}^{\bar{}\ \breve{}}$ for $a'=a_-,a$. If $\bar{W}_a = \{0\}$, then $\bar \Psi_a$ is empty.

For each knot $a$, define
$$
c_a:=
\left[
\begin{matrix}
c^{\bar{}~\breve{}}_{aa_-}&c^{\bar{}~\bar{}}_{aa}&c^{\bar{}~\breve{}}_{aa}\\
0&0&c^{\breve{}~\breve{}}_{aa}
\end{matrix}
\right] \quad \text{and} \quad d_a:=
\left[
\begin{matrix}
d^{\bar{}~\breve{}}_{aa_-}&d^{\bar{}~\bar{}}_{aa}&d^{\bar{}~\breve{}}_{aa}\\
0&0&d^{\breve{}~\breve{}}_{aa}
\end{matrix}
\right].
$$
By equations (\ref{ceq}) and (\ref{deq}) we have that
$$
\left[
\begin{matrix}
\bar \Phi_a^0\\
\breve \Phi_a^0
\end{matrix}
\right]=
c_a
\left[
\begin{matrix}
\breve \Phi_{a_-}^1\\
\bar \Phi_a^1\\
\breve \Phi_a^1
\end{matrix}
\right]
~\text{and}~
\left[
\begin{matrix}
\bar \Psi_a\\
\breve \Psi_a
\end{matrix}
\right]=
d_a
\left[
\begin{matrix}
\breve \Phi_{a_-}^1\\
\bar \Phi_a^1\\
\breve \Phi_a^1
\end{matrix}
\right].
$$
It follows from Corollary \ref{WaveCor} that 
$$
V_a^1 = A_{a_-}^+ \oplus V_a^0 \oplus W_a \oplus A_{a_+}^-.
$$
Since
$$
\left[
\begin{matrix}
\alpha_{a_-}^+\\
\breve \Phi_{a_-}^0\\
\bar \Phi_a^0\\
\breve \Phi_a^0\\
\breve \Psi_{a_-}\\
\bar \Psi_a\\
\breve \Psi_a\\
\alpha_{a_+}^-
\end{matrix}
\right]=
\left[
\begin{matrix}
b^{\bar{}~\breve{}}_{a_-a_-}&0&0\\
c^{\breve{}~\breve{}}_{a_-a_-}&0&0\\
c^{\bar{}~\breve{}}_{aa_-}&c^{\bar{}~\bar{}}_{aa}&c^{\bar{}~\breve{}}_{aa}\\
0&0&c^{\breve{}~\breve{}}_{aa}\\
d^{\breve{}~\breve{}}_{a_-a_-}&0&0\\
d^{\bar{}~\breve{}}_{aa_-}&d^{\bar{}~\bar{}}_{aa}&d^{\bar{}~\breve{}}_{aa}\\
0&0&d^{\breve{}~\breve{}}_{aa}\\
0&0&b^{\bar{}~\breve{}}_{a_+a}
\end{matrix}
\right]
\left[
\begin{matrix}
\breve \Phi_{a_-}^1\\
\bar \Phi_a^1\\
\breve \Phi_a^1
\end{matrix}
\right]
$$
it follows that the matrix
$$
\left[
\begin{matrix}
b^{\bar{}~\breve{}}_{a_-a_-}&0&0\\
c^{\breve{}~\breve{}}_{a_-a_-}&0&0\\
c^{\bar{}~\breve{}}_{aa_-}&c^{\bar{}~\bar{}}_{aa}&c^{\bar{}~\breve{}}_{aa}\\
0&0&c^{\breve{}~\breve{}}_{aa}\\
d^{\breve{}~\breve{}}_{a_-a_-}&0&0\\
d^{\bar{}~\breve{}}_{aa_-}&d^{\bar{}~\bar{}}_{aa}&d^{\bar{}~\breve{}}_{aa}\\
0&0&d^{\breve{}~\breve{}}_{aa}\\
0&0&b^{\bar{}~\breve{}}_{a_+a}
\end{matrix}
\right]=
\left[
\begin{matrix}
b^{\bar{}~\breve{}}_{a_-a_-}&0&0\\
c^{\breve{}~\breve{}}_{a_-a_-}&0&0\\
{}&c_a&{}\\
d^{\breve{}~\breve{}}_{a_-a_-}&0&0\\
{}&d_a&{}\\
0&0&b^{\bar{}~\breve{}}_{a_+a}
\end{matrix}
\right]
$$
is orthogonal.

We can generalize this orthogonal matrix as follows. Let $a<b$ be knots so that $[a,b]$ contains at least three knots. We define matrices $c_{[a,b]}$ inductively:
$$
c_{[a_-,a_+]}:= 
\left[
\begin{matrix}
c^{\breve{}~\breve{}}_{a_-a_-}&0&0\\
{}&c_a&{}
\end{matrix}
\right].
$$
Suppose $[a,b]$ contains more than three knots. Then 
$$
c_{[a,b]}:=
\left[
\begin{matrix}
c_{[a,b_-]}&0&0\\
{}&0 \dots c_{b_-}&{}
\end{matrix}
\right].
$$
Note that the above notation indicates that the second row of the block array consists of zeros followed by $c_{b_-}$ at the end. For example
$$
c_{[a_-,a_+]}=
\left[
\begin{matrix}
c^{\breve{}~\breve{}}_{a_-a_-}&0&0\\
c^{\bar{}~\breve{}}_{aa_-}&c^{\bar{}~\bar{}}_{aa}&c^{\bar{}~\breve{}}_{aa}\\
0&0&c^{\breve{}~\breve{}}_{aa}
\end{matrix}
\right].
$$
For four consecutive knots $a<b<c<d$,
$$
c_{[a,d]}=
\left[
\begin{matrix}
c^{\breve{}~\breve{}}_{aa}&0&0&0&0\\
c^{\bar{}~\breve{}}_{ba}&c^{\bar{}~\bar{}}_{bb}&c^{\bar{}~\breve{}}_{bb}&0&0\\
0&0&c^{\breve{}~\breve{}}_{bb}&0&0\\
0&0&c^{\bar{}~\breve{}}_{cb}&c^{\bar{}~\bar{}}_{cc}&c^{\bar{}~\breve{}}_{cc}\\
0&0&0&0&c^{\breve{}~\breve{}}_{cc}
\end{matrix}
\right].
$$

We similarly define matrices $d_{[a,b]}$ inductively by
$$
d_{[a_-,a_+]}:= 
\left[
\begin{matrix}
d^{\breve{}~\breve{}}_{a_-a_-}&0&0\\
{}&d_a&{}
\end{matrix}
\right], 
$$
and, if $[a,b]$ contains more than three knots, then 
$$
d_{[a,b]}:=
\left[
\begin{matrix}
d_{[a,b_-]}&0&0\\
{}&0 \dots d_{b_-}&{}
\end{matrix}
\right].
$$
Next, we use these matrices to define
$$
M_{[a,b]}:=
\left[
\begin{matrix}
b^{\bar{}~\breve{}}_{aa}~0\dots0\\
c_{[a,b]}\\
d_{[a,b]}\\
0\dots0b^{\bar{}~\breve{}}_{bb_-}
\end{matrix}
\right]
$$
It follows that
$$
\left[
\begin{matrix}
\alpha_a^+\\
\breve \Phi_a^0\\
\bar \Phi_{a_+}^0\\
\breve \Phi_{a_+}^0\\
\vdots\\
\breve \Phi_{b_-}^0\\
\breve \Psi_a\\
\bar \Psi_{a_+}\\
\breve \Psi_{a_+}\\
\vdots\\
\breve \Psi_{b_-}\\
\alpha_b^-
\end{matrix}
\right]=
M_{[a,b]}
\left[
\begin{matrix}
\breve \Phi_a^1\\
\bar \Phi_{a_+}^1\\
\breve \Phi_{a_+}^1\\
\vdots\\
\breve \Phi_{b_-}^1
\end{matrix}
\right],
$$
and thus $M_{[a,b]}$ is an orthogonal matrix.

We now use the results of Section \ref{dws} to decompose the matrix
$$
\left[
\begin{matrix}
d^{\bar{}~\breve{}}_{aa_-}&d^{\bar{}~\bar{}}_{aa}&d^{\bar{}~\breve{}}_{aa}
\end{matrix}
\right]
$$
when it is defined.
Using the fact that $\hat W_a =(I-P_{\bar V_a^0})\bar V_a^1$ (see Lemma \ref{stepone}), we obtain that
$$
\left[
\begin{matrix}
-(c^{\bar{}~\bar{}}_{aa})^T(c^{\bar{}~\breve{}}_{aa_-})&(I-(c^{\bar{}~\bar{}}_{aa})^Tc^{\bar{}~\bar{}}_{aa})&-(c^{\bar{}~\bar{}}_{aa})^T(c^{\bar{}~\breve{}}_{aa})
\end{matrix}
\right]
\left[
\begin{matrix}
\breve \Phi_{a_-}^1\\
\bar \Phi_a^1\\
\breve \Phi_a^1
\end{matrix}
\right]
$$
is a basis for $\hat W_a$. Let $\hat g_a$ be a matrix with orthonormal rows with the same row span as
$$
\left[
\begin{matrix}
-(c^{\bar{}~\bar{}}_{aa})^T(c^{\bar{}~\breve{}}_{aa_-})&(I-(c^{\bar{}~\bar{}}_{aa})^Tc^{\bar{}~\bar{}}_{aa})&-(c^{\bar{}~\bar{}}_{aa})^T(c^{\bar{}~\breve{}}_{aa})
\end{matrix}
\right].
$$
Write
$$
\hat g_a=
\left[
\begin{matrix}
\hat g^{\bar{}~\breve{}}_{aa_-}&\hat g^{\bar{}~\bar{}}_{aa}&\hat g^{\bar{}~\breve{}}_{aa}
\end{matrix}
\right].
$$
Then
$$
\left[
\begin{matrix}
\hat g^{\bar{}~\breve{}}_{aa_-}&\hat g^{\bar{}~\bar{}}_{aa}&\hat g^{\bar{}~\breve{}}_{aa}
\end{matrix}
\right]
\left[
\begin{matrix}
\breve \Phi_{a_-}^1\\
\bar \Phi_a^1\\
\breve \Phi_a^1
\end{matrix}
\right]
$$
is an orthonormal basis for $\hat W_a$.
Also, from Lemma \ref{tildew}, we have that $\tilde W_a =(I-P_{\bar V_a^0 \oplus \hat W_a})(A_a^- \oplus A_a^+)$. This implies that
$$
\left[
\begin{matrix}
b^{\bar{}~\breve{}}_{aa_-}&0&0\\
0&0&b^{\bar{}~\breve{}}_{aa}
\end{matrix}
\right]
\left(
I-
\left[
\begin{matrix}
c^{\bar{}~\breve{}}_{aa_-}&c^{\bar{}~\bar{}}_{aa}&c^{\bar{}~\breve{}}_{aa}\\
\hat g^{\bar{}~\breve{}}_{aa_-}&\hat g^{\bar{}~\bar{}}_{aa_-}&\hat g^{\bar{}~\breve{}}_{aa}
\end{matrix}
\right]^T
\left[
\begin{matrix}
c^{\bar{}~\breve{}}_{aa_-}&c^{\bar{}~\bar{}}_{aa}&c^{\bar{}~\breve{}}_{aa}\\
\hat g^{\bar{}~\breve{}}_{aa_-}&\hat g^{\bar{}~\bar{}}_{aa_-}&\hat g^{\bar{}~\breve{}}_{aa}
\end{matrix}
\right]
\right)
\left[
\begin{matrix}
\breve \Phi_{a_-}^1\\
\bar \Phi_a^1\\
\breve \Phi_a^1
\end{matrix}
\right]
$$
is a basis for $\tilde W_a$. Let $\tilde g_a$ be a matrix with orthonormal rows and with the same row space as 
$$
\left[
\begin{matrix}
b^{\bar{}~\breve{}}_{aa_-}&0&0\\
0&0&b^{\bar{}~\breve{}}_{aa}
\end{matrix}
\right]
\left(
I-
\left[
\begin{matrix}
c^{\bar{}~\breve{}}_{aa_-}&c^{\bar{}~\bar{}}_{aa}&c^{\bar{}~\breve{}}_{aa}\\
\hat g^{\bar{}~\breve{}}_{aa_-}&\hat g^{\bar{}~\bar{}}_{aa_-}&\hat g^{\bar{}~\breve{}}_{aa}
\end{matrix}
\right]^T
\left[
\begin{matrix}
c^{\bar{}~\breve{}}_{aa_-}&c^{\bar{}~\bar{}}_{aa}&c^{\bar{}~\breve{}}_{aa}\\
\hat g^{\bar{}~\breve{}}_{aa_-}&\hat g^{\bar{}~\bar{}}_{aa_-}&\hat g^{\bar{}~\breve{}}_{aa}
\end{matrix}
\right]
\right)
$$
Write
$$
\tilde g_a=
\left[
\begin{matrix}
\tilde g^{\bar{}~\breve{}}_{aa_-}&\tilde g^{\bar{}~\bar{}}_{aa}&\tilde g^{\bar{}~\breve{}}_{aa}
\end{matrix}
\right].
$$
Then
$$
\left[
\begin{matrix}
\tilde g^{\bar{}~\breve{}}_{aa_-}&\tilde g^{\bar{}~\bar{}}_{aa}&\tilde g^{\bar{}~\breve{}}_{aa}
\end{matrix}
\right]
\left[
\begin{matrix}
\breve \Phi_{a_-}^1\\
\bar \Phi_a^1\\
\breve \Phi_a^1
\end{matrix}
\right]
$$
is an orthonormal basis for $\tilde W_a$.
Thus we have that
$$
\left[
\begin{matrix}
d^{\bar{}~\breve{}}_{aa_-}&d^{\bar{}~\bar{}}_{aa}&d^{\bar{}~\breve{}}_{aa}
\end{matrix}
\right]=
\left[
\begin{matrix}
\hat g^{\bar{}~\breve{}}_{aa_-}&\hat g^{\bar{}~\bar{}}_{aa}&\hat g^{\bar{}~\breve{}}_{aa}\\
\tilde g^{\bar{}~\breve{}}_{aa_-}&\tilde g^{\bar{}~\bar{}}_{aa}&\tilde g^{\bar{}~\breve{}}_{aa}
\end{matrix}
\right].
$$

\section{Efficient construction and applications}\label{efficient_constr}
\subsection{Efficient construction of scaling and wavelet matrix coefficients for  the example of Section~\ref{ex1}}

\subsubsection{Knot sequences on a closed interval with corresponding orthogonal bases.}\label{knot_seq}

Let $\mu < \nu $ be real numbers. Let $\mathbf{a}$ be a knot sequence on $[\mu,\nu]$ that includes $\mu$ and $\nu$. (Before continuing we recall that if $\mathbf{a}$ is a knot sequence and $a \in  \mathbf{a}$, then $(\mathbf{a},a_-)$ (resp. $(\mathbf{a},a_+)$) denotes the predecessor (resp. successor) of $a$ relative to the knot sequence $\mathbf{a}$.  For each knot $a \in \mathbf{a}$ with $\mu \le a <\nu$, we assume there is a corresponding $\theta_{(\mathbf{a},a)} \in (0,1)$;  $\theta_{(\mathbf{a},a)} \in (0,1)$ depends on the knot sequence and the knot. As in Section~\ref{ex1}, if we let $\boldsymbol{\theta}$ be the sequence $(\theta_{(\mathbf{a},a)})_{\mu \le a <\nu}$, then we can define the orthogonal basis $\Omega_{\mathbf{a},\boldsymbol{\theta}}$.  The component functions of $\bar\Omega_{(\mathbf{a},a)}$ (resp. $\breve\Omega_{(\mathbf{a},a)}$), for suitably chosen knots $a$, can be normalized to produce $\bar\Phi_{(\mathbf{a},a)}$ (resp. $\breve\Phi_{(\mathbf{a},a)}$).

\subsubsection{Knot sequence refinements obtained by adding a single knot.}\label{insert_one_knot}

Let $\mu < \nu $ be real numbers. Suppose given  a knot sequence $\mathbf{a}^0$  on $[\mu,\nu]$ and a corresponding sequence $\boldsymbol{\theta}^0 = (\theta_{(\mathbf{a}^0,a)})_{a \in \mathbf{a}^0, a < \nu}$   as described in Section \ref{knot_seq}. We will now consider a simple refinement of $\mathbf{a}^0$ obtained by adding a single new knot. 

Fix $a \in \mathbf{a}^0$ so that $a < \nu$. Let $\mathbf{a}^1:= \mathbf{a}^0 \cup \{ b^0_a\}$, where $b^0_a$ is as defined in Section \ref{ex1}. Choose $\theta', \theta'' \in (0,1)$. We define a new sequence $\boldsymbol{\theta}^1 = (\theta_{(\mathbf{a}^1,a)})_{a \in \mathbf{a}^1, a < \nu}$ as follows: if $a' \in \mathbf{a}^0$, $a'< \nu$, and either $a' < a$ or $a' \ge (\mathbf{a}^0,a_+)$, then, by definition, $a' \in \mathbf{a}^1$ and we let $\theta_{(\mathbf{a}^1,a')}:=\theta_{(\mathbf{a}^0,a')}$; $\theta_{(\mathbf{a}^1,a)}:= \theta'$, and $\theta_{(\mathbf{a}^1,a_+)}:= \theta''$. It follows from Lemma \ref{c0quadlemma} that
$$S(\Omega_{\mathbf{a^{0}},\boldsymbol{\theta
^{0}}})\subset S(\Omega_{\mathbf{a^{1}},\boldsymbol{\theta^{1}}}).$$

As in Section \ref{knot_seq}, for $\epsilon \in \{0,1\}$ let $\Phi_{(\mathbf{a}^{\epsilon},\cdot)}$ be the orthonormal basis for $S(\Omega_{\mathbf{a^{\epsilon}},\boldsymbol{\theta
^{\epsilon}}})$.  We remark that  $\Phi^0:=\Phi_{(\mathbf{a}^{0},\cdot)}$ and $\Phi^1:=\Phi_{(\mathbf{a}^{1},\cdot)}$ are both orthonormal bases centered on $\mathbf{a}^0$ and  are referenced according to $\mathbf{a}^0$.  
If $ (\mathbf{a}^0,a_{+})<\nu$, then $\Phi_{(\mathbf{a}^{1},\cdot)}$ is obtained from $\Phi_{(\mathbf{a}^{0},\cdot)}$ by replacing the four functions represented by  $\Phi_{(\mathbf{a}^0,a)}$ and $\bar \Phi_{(\mathbf{a}^0,a_+)}$ by the seven functions represented by $\Phi_{(\mathbf{a}^1,a)}$, $\Phi_{(\mathbf{a}^1,a_+)}$, and $\bar\Phi_{(\mathbf{a}^1,a_{++})}$. If $(\mathbf{a}^0,a_+)=\nu$ then $\Phi_{(\mathbf{a}^{1},\cdot)}$ is obtained from $\Phi_{(\mathbf{a}^{0},\cdot)}$ by replacing the four functions in $\Phi_{(\mathbf{a}^0,a)}$ by the seven functions represented by $\Phi_{(\mathbf{a}^1,a)}$ and $\Phi_{(\mathbf{a}^1,a_+)}$. 
It follows from Theorem \ref{ThmWavelets} that there exists an orthonormal basis $\Psi$ centered on $\mathbf{a}^0$ so that $S(\Phi_{(\mathbf{a}^{1},\cdot)})= S(\Phi_{(\mathbf{a}^{0},\cdot)}) \oplus S(\Psi)$. As in that theorem, we let $W=S(\Psi)$.  From the preceding paragraph we get that $\dim W =3$.  It is easy to check that all the spaces $W_{a'}$ are trivial except when $a'=a$ or $(\mathbf{a}^0,a_+)$. 

We can now find the scaling matrix coefficients described in Section \ref{MatCoeffs}. As in that section, the wavelet matrix coefficients can then be easily constructed. 
We briefly describe the case where $\mu < a< (\mathbf{a}^0,a_{+})<(\mathbf{a}^0,a_{++})<\nu$, the other cases being similar. The scaling matrix coefficients involved in computing the wavelet matrix coefficients are $c_{aa},c_{(\mathbf{a}^0,a_{+})a}$, and $c_{(\mathbf{a}^0,a_{+})(\mathbf{a}^0,a_{+})}$. $c_{aa}$ and $c_{(\mathbf{a}^0,a_{+})a}$ are $3 \times 6$ matrices, and $c_{(\mathbf{a}^0,a_{+})(\mathbf{a}^0,a_{+})}$ is a $3 \times 3$ diagonal matrix with the lower right $2 \times 2$ block being the identity matrix.

Using the dimension notation of Section \ref{dws}, it is easy to check that $\bar k^0_a=\bar k^0_{a_+}=\bar k^1_a=\bar k^1_{a_+}=1$, $m^+_a=m^-_{a_+}=0$, $m^-_a=m^+_{a_+}=1$, $\breve k^0_a=2$, and $\breve k^1_a=5$.  From Theorem~\ref{waveletdim}  we obtain  $\dim \bar W_a=\dim \bar W_{a_+}=\dim \breve W_a=1$. Using the methods of Section \ref{MatCoeffs} we get the $1 \times 1$ matrices $d^{~\bar{}~\bar{}}_{aa}$ and $d^{~\bar{}~\bar{}}_{a_+a_+}$, the $1 \times 5$ matrices $d^{~\bar{}~\breve{}}_{aa}$ and $d^{~\bar{}~\breve{}}_{a_+a}$, and the $1 \times 5$ matrix $d^{~\breve{}~\breve{}}_{aa}$.

\subsection{Greedy Algorithm}

\subsubsection{Dropping one knot}\label{dropknot}

Assume $\mathbf{a}^1$ is a knot sequence on $[\mu,\nu]$ as discussed in section \ref{insert_one_knot}. Let $b \in \mathbf{a}^1$ so that $\mu < b < \nu$; i.e. $b$ is an interior knot of $\mathbf{a}^1$. Choose $l,m \in \mathbf{a}^1$ so that $l,b$ and $m$ are consecutive knots in $\mathbf{a}^1$.
Next we let $\mathbf{a}^0 = \mathbf{a}^1 \setminus \{ b\}$.  If $a \in \mathbf{a}^1$ and either $a < l$ or $a \ge m$, then clearly $a\in \mathbf{a}^0$, and in these cases we define $\theta_{(\mathbf{a}^0,a)}:=\theta_{(\mathbf{a}^1,a)}$. Also, we define
$$
\theta_{(\mathbf{a}^0,l )}:= \frac{b-l}{m-l}.
$$

Let $f \in L^2([\mu,\nu])$. For $i \in \{0,1 \}$, let $c^i_a := \langle \Phi_{(\mathbf{a}^i,a)},f\rangle$, $\bar c^i_a := \langle \bar\Phi_{(\mathbf{a}^i,a)},f\rangle$, and $\breve c^i_a := \langle \breve\Phi_{(\mathbf{a}^i,a)},f\rangle$. Since $\mathbf{a}^1$ is a refinement of $\mathbf{a}^0$, then we may also use the wavelet basis $(\Psi_a)_{a \in \mathbf{a}^0, a < \nu}$ to define $d_a := \langle \Psi_{a},f\rangle$, $\bar d_a := \langle \bar\Psi_{a},f\rangle$, and $\breve d_a := \langle \breve\Psi_{a},f\rangle$.
If $a<l$ or $m<a<\nu$, then $c^0_a=c^1_a$. Also, if $m<\nu$ we have $\breve c^0_m=\breve c^1_m$. As we shall see below, the  coefficient sequence $(c^0_a)_{a\in \mathbf{a}^0, a < \nu}$ is obtained from $(c^1_a)_{a\in \mathbf{a}^1, a < \nu}$ by a fairly simple replacement procedure.

We briefly describe the case where  $\mu < l < b< m < \nu$, the other cases being similar.

$$
c^0_l = 
\left[
\begin{matrix}
c_{ll}^{~\bar{}~ \bar{}}&c_{ll}^{~\bar{}~ \breve{}}\\
0&c_{ll}^{~\breve{}~ \breve{}}
\end{matrix}
\right]
\left[
\begin{matrix}
c^1_l\\
c^1_b
\end{matrix}
\right],
\bar c^0_m = c_{ml}^{~\bar{}~ \breve{}}
\left[
\begin{matrix}
\breve c^1_l\\
c^1_b
\end{matrix}
\right]
+ c_{mm}^{~\bar{}~ \bar{}} \bar c^1_m,
$$
and

$$
d_l = 
\left[
\begin{matrix}
d_{ll}^{~\bar{}~ \bar{}}&d_{ll}^{~\bar{}~ \breve{}}\\
0&d_{ll}^{~\breve{}~ \breve{}}
\end{matrix}
\right]
\left[
\begin{matrix}
c^1_l\\
c^1_b
\end{matrix}
\right],
\bar d_m = d_{ml}^{~\bar{}~ \breve{}}
\left[
\begin{matrix}
c^1_l\\
c^1_b
\end{matrix}
\right]
+ d_{mm}^{~\bar{}~ \bar{}} \bar c^1_m.
$$

Here,  
$
\left[
\begin{matrix}
c_{ll}^{~\bar{}~ \bar{}}&c_{ll}^{~\bar{}~ \breve{}}\\
0&c_{ll}^{~\breve{}~ \breve{}}
\end{matrix}
\right]
$
is a $3 \times 6$ matrix,  $c_{ml}^{~\bar{}~ \breve{}}$ is a $1 \times 5$ matrix,  $c_{mm}^{~\bar{}~ \bar{}}$ is a $1 \times 1$ matrix,  
$
\left[
\begin{matrix}
d_{ll}^{~\bar{}~ \bar{}}&d_{ll}^{~\bar{}~ \breve{}}\\
0&d_{ll}^{~\breve{}~ \breve{}}
\end{matrix}
\right]
$
is a $2 \times 6$ matrix, $d_{ml}^{~\bar{}~ \breve{}}$ is a $1 \times 5$ matrix, and $d_{mm}^{~\bar{}~ \bar{}}$ is a $1 \times 1$ matrix.
Note that in each case seven scalar scaling coefficients at level 1 are ``replaced'' by four scalar scaling coefficients at level 0 and three scalar wavelet coefficients.

\subsubsection{Greedy algorithm.}

Assume $\mathbf{a}$ is a knot sequence on $[\mu,\nu]$ as discussed in the previous subsection, and $f \in L^2([\mu,\nu])$. For each interior knot of $\mathbf{a}$ we drop that knot, as in Section \ref{dropknot}, and compute the triple of wavelet coefficients that result. We choose the interior knot that minimizes the $l^2$ norm of the triple. Then we remove that knot and repeat the procedure on the resulting new knot sequence. This is done until all original interior knots have been dropped. Below are the details of this algorithm. 

Let $b$ be an interior knot of $\mathbf{a}$ . $\text{Drop}(\mathbf{a},b)$ will denote the knot sequence obtained from $\mathbf{a}$ by dropping $b$, as described in Section \ref{dropknot}. Let $f \in L^2([\mu,\nu])$, and $a \in \mathbf{a}$ with $a<\nu$. $\text{Coeff}(\mathbf{a},f,a) := \langle \Phi_{(\mathbf{a},a)},f\rangle$. $\text{Coeff}(\mathbf{a},f)$ will denote the sequence $\big(\text{Coeff}(\mathbf{a},f,a)\big)_{a<\nu}$.
Next let $c=\text{Coeff}(\mathbf{a},f)$ for some $f$ and let $b$ be an interior knot of $\mathbf{a}$. If we temporarily let $\mathbf{a}^1:= \mathbf{a}$ and $\mathbf{a}^0:= \text{Drop}(\mathbf{a},b)$, as in Section \ref{dropknot}, then let $\text{Wave}(\mathbf{a},c,b)$ be the triple of wavelet coefficients that results from dropping knot $b$. $\text{Wave}(\mathbf{a},c)$ will denote the sequence $\big(\text{Wave}(\mathbf{a},c,b)\big)_{\mu<b<\nu}$.
If $(d_b)_{\mu<b<\nu}$ is a sequence of triples indexed by the interior knots of $\mathbf{a}$, $\text{MinKnot}(\mathbf{a},d)$ will denote the interior knot with the smallest $l^2$ norm of its triple.

The algorithm listed below will produce a sequence of knot sequences by successively dropping knots. Suppose $\mathbf{a}$ has $N$ interior knots.

\begin{enumerate}
\item $\mathbf{a}^N=\mathbf{a}$.
\item For $i=N,N-1,\dots,1$,
\begin{enumerate}
\item $c^i=\text{Coeff}(\mathbf{a}^i,f)$;
\item $d^i=\text{Wave}(\mathbf{a^i},c^i)$;
\item $b^i= \text{MinKnot}(\mathbf{a}^i,d^i)$;
\item $\mathbf{a}^{i-1}=\text{Drop}(\mathbf{a^i},b^i)$.
\end{enumerate}

\end{enumerate}

\subsubsection{Data Example}

We now apply the greedy algorithm to an example data set. We consider data from one row of an $200\times 200$ gray scale image (Example/ocelot.jpg) available from the \emph{Mathematica} TestImage image library.   
We extract from this image the  first 199 entries of the $120^{\text{th}}$ row. Figure \ref{Image_data} shows a linearly interpolated plot of this data. 

\begin{figure}[h]
\begin{center}
\scalebox{0.6}{\includegraphics{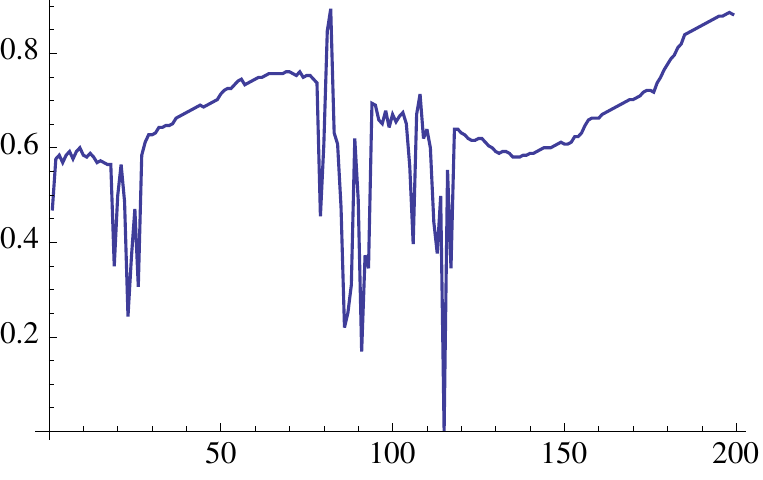}}
\caption{Linearly interpolated samples from a row of the $200\times 200$ gray scale image (Example/ocelot.jpg) in the \emph{Mathematica} TestImage library.} \label{Image_data}
\end{center}
\end{figure}

We next consider the knot sequence on $[1,199]$ given by $\mathbf{a}:=(1,4,7,\dots,196,199)$. For each knot $a \in \mathbf{a}$ with $1 \le a <199$, let $\theta_{(\mathbf{a},a)}= 1/2$. Note that for $1 \le a < 196$, $\Phi_{(\mathbf{a},a)}$ consists of three orthogonal functions, and $\Phi_{(\mathbf{a},196)}$ consists of four orthogonal functions. It follows that $\dim S(\Phi_{(\mathbf{a},\cdot)})=199$. It is easy to check that there is a unique function $f \in S(\Phi_{(\mathbf{a},\cdot)})$ that interpolates the data, i.e. $f(i)=d_i$ for $1 \le i \le 199$.

We now show the results of applying the greedy algorithm to the function $f$. Note that $\mathbf{a}$ has 65 interior knots. Thus in step $(1)$, we let $N=65$ and $\mathbf{a}^{65} := \mathbf{a}$. For $1 \le i \le 65$,  
$$
\Vert f - P_{S(\Phi_{(\mathbf{a^i},\cdot)})}f \Vert^2
$$
represents the square of the error obtained by approximating $f$ with $P_{S(\Phi_{(\mathbf{a^i},\cdot)})}f$. This function of $i$ is plotted in Figure \ref{Error_data}.

\begin{figure}[h]
\begin{center}
\scalebox{0.4}{\includegraphics{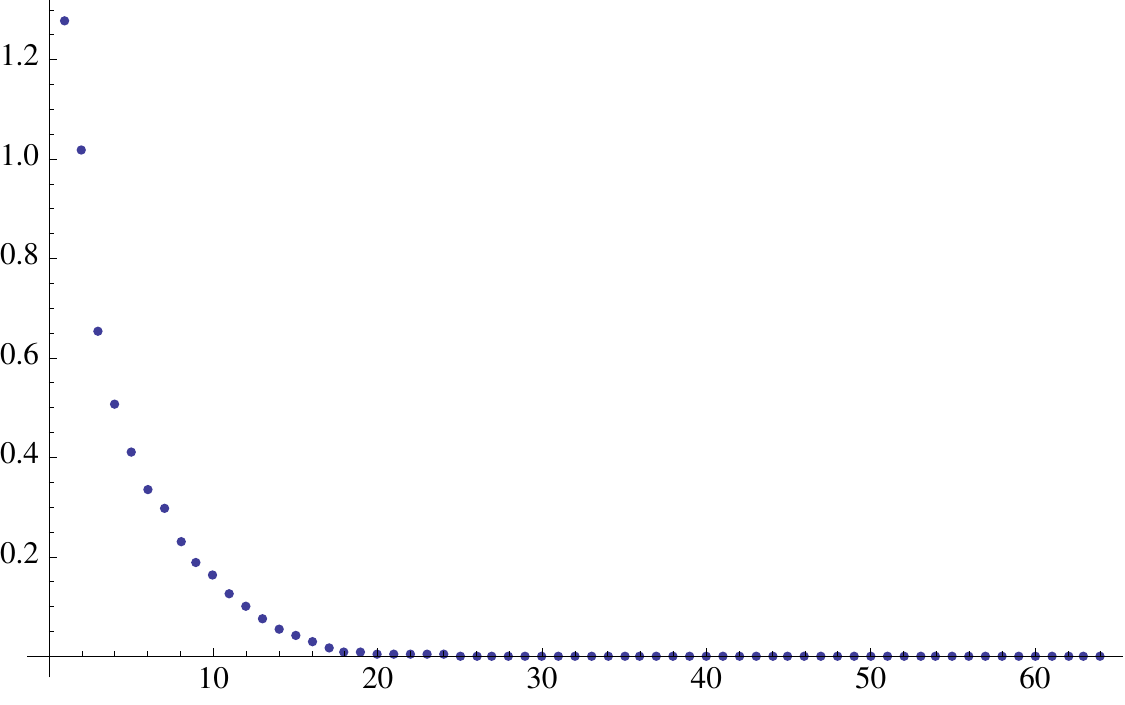}}
\caption{Error Data} \label{Error_data}
\end{center}
\end{figure}

For example, if $i=20$ the value of the square of the error is $0.00491487$. $\mathbf{a}^{20}$ is a knot sequence with 20 interior knots, and is the result of using the greedy algorithm to drop 45 knots, one at a time, from the original knot sequence $\mathbf{a}$ which has 65 interior knots. In Figure  \ref{Compressed_data} we see the plot of $P_{S(\Phi_{(\mathbf{a^{20}},\cdot)})}f$, the plot of the original data set, and the knots in $\mathbf{a^{20}}$  on the horizontal axis.

\begin{figure}[h]
\begin{center}
\scalebox{0.25}{\includegraphics{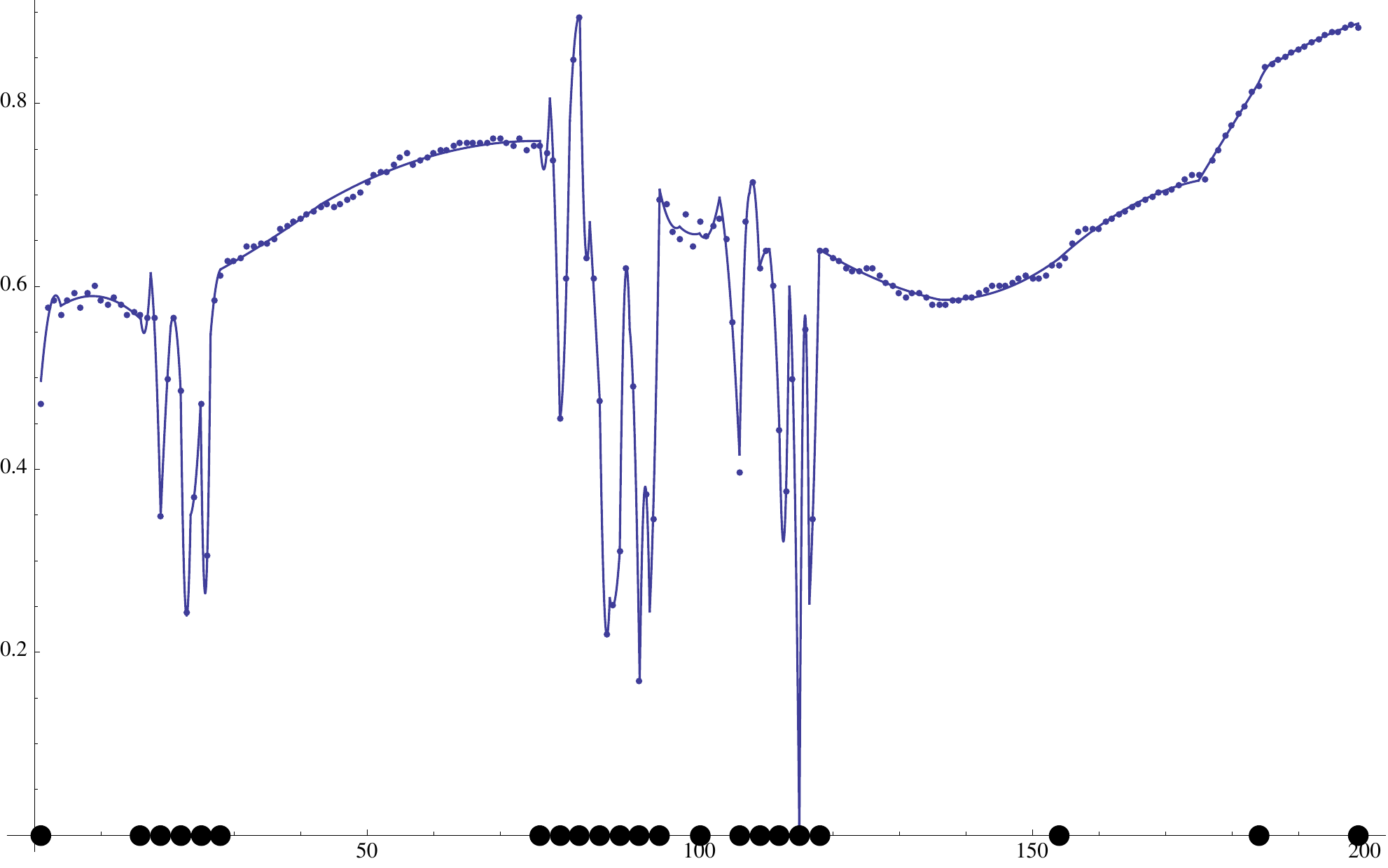}}
\caption{Continuous plot of projection with original data points and knots of $\mathbf{a^{20}}$ on the horizontal axis.} \label{Compressed_data}
\end{center}
\end{figure}

\subsection{Arbitrary polynomial reproduction: Wavelet construction}

For a knot sequence $\mathbf{a}$ and $n=1,2,\ldots$, let $V^0:=S(\Omega_{\mathbf{a}}^n)$ 
and $V^1:=S(\Omega_{\mathbf{a}}^{n+3})$ where $\Omega_{\mathbf{a}}^n$ is given in  Section~\ref{ex1.2}.  By Theorem~\ref{ThmWavelets} there is a basis $\Psi_{\mathbf{a}}^n$ centered on $\mathbf{a}$ such that
$V^1=V^0\oplus W$ where $W=S(\Psi_\mathbf{a}^n)$.  We next use the algorithm from section \ref{waveletConst} to construct such a  $\Psi_{\mathbf{a}}^n$.

Recall that 
$$
r^n =(I-P_{\{\tilde\phi^2,\ldots,\tilde\phi^n,z^n\}})r=(I-P_{\{z^n\}})r_n $$ and 
$$
l^n=(I-P_{\{\tilde\phi^2,\ldots,\tilde\phi^n, z^n\}})l=(I-P_{\{z^n\}})l_n .
$$
For $n=1,2,\ldots,$ let 
$$
\Lambda^n:=\text{span}\{ \tilde\phi^2, \ldots, \tilde\phi^n, z^n\}. 
$$
Note that $\Lambda^1=\text{span}\{z^1\}$.
From the definitions of $\tilde\phi^n$ and $z^n$ it follows that
$$
\Lambda^{n+3}=\Lambda^n\oplus \text{span}\{\tilde\phi^{n+2},z^n_{\perp},z^{n+3}\}=\Lambda^n\oplus\Delta^n,$$
where 
$$z^n_{\perp}=\frac{\tilde\phi^{n+1}}{||\tilde\phi^{n+1}||^2}-\alpha_n\frac{\tilde\phi^{n+3}}{||\tilde\phi^{n+3}||^2}
$$
and $$
\Delta^n:=\text{span}\{\tilde\phi^{n+2},z^n_{\perp},z^{n+3}\}.
$$
Clearly, $z^n_{\perp}$ and $z^n$ are orthogonal and span the same space as $\tilde\phi^{n+1}$ and  $\tilde\phi^{n+3}$.

Now $$\bar V_a^0=\text{span }   \bar\Omega^n_a =\text{span} \{ \bar\omega^n_a \}$$
where $\bar\omega^n_a=r^n\circ\sigma_{a_-}+l^n\circ\sigma_a$.
Similarly $$\bar V_a^1=\text{span }   \bar\Omega^{n+3}_a =\text{span} \{ \bar\omega^{n+3}_a \}$$
where $\bar\omega^{n+3}_a=r^{n+3}\circ\sigma_{a_-}+l^{n+3}n\circ\sigma_a$.
It follows that $\bar k_a^0 = \bar k_a^1=1$ for every $a$. Furthermore tt is easy to see that for every $a$, $Y_a = Y_a^- = Y_a^+ = \{0 \}$, and so $m_a=m_a^-=m_a^+=0$. Lemmas \ref{stepone} and \ref{tildew} imply that $\dim \hat W_a^n = \dim \tilde W_a^n = 1$. The superscript denotes the dependence on $n$ in the construction.
From the definitions of $r^n$ and $l^n$, and the symmetries of the polynomials $\tilde \phi^n$, it can be shown that $l^n$ is the reflection of $r^n$ with respect to the line $x=1/2$. From this we can see that 
\begin{align}
\langle \bar\omega^n_a, \bar\omega^n_a \rangle  & = (a_+ - a_-)\langle r^n,r^n \rangle \notag \\
\langle \bar\omega^n_a, \bar\omega^{n+3}_a \rangle  & = (a_+ - a_-)\langle r^n,r^{n+3} \rangle \notag 
\end{align}
Since $\Lambda^{n+3}=\Lambda^n \oplus \Delta^n$ and $r^n=r-P_{\Lambda^n}r$, we get 
\begin{equation}\label{r_diff}
r^n-r^{n+3}=P_{\Delta^n}r.
\end{equation}
Now, $r^{n+3} \perp \Delta^n$. From this it follows that 
$$
\langle  r^n  , r^{n+3} \rangle=\langle  r^{n+3}  , r^{n+3} \rangle. 
$$

We now begin the construction of the wavelets. Since $\hat W_a^n$ is one dimensional  equation~(\ref{wavehat}) shows it is
spanned by the single function 
$\hat w_a^n := (I-P_{\bar w_a^n})\bar\omega^{n+3}_a$ which is,
$$
\hat w_a^n=(r^{n+3}-\frac{\langle  r^{n+3}  , r^{n+3} \rangle}{\langle  r^n  , r^n \rangle }r^n) \circ \sigma_{a_-}+ (l^{n+3}-\frac{\langle  r^{n+3}  , r^{n+3} \rangle}{\langle  r^n  , r^n \rangle }l^n) \circ \sigma_a.
$$
Since $T_a^n  = (P_{\breve V_{a_-}^1}+P_{\breve V_a^1})\bar V_a^0$ it to is
one dimensional and 
by Lemma \ref{ydim} it follows that $T_a^- = T_a^+ = \{0\}$, and so $U_a^n
= T_a^n$ hence $S_a^n$ is spanned by the single function
$$
s_a^n = -(r^n-r^{n+3}) \circ \sigma_{a_-}+ (l^n-  l^{n+3}) \circ \sigma_a.
$$
Equation~(\ref{wavetilde}) shows that  $\tilde W_a^n$ is spanned by  $(I-P_{\bar V_a^0 \oplus \hat
  W_a^n})s_a^n$. A straightforward computation shows that this function is a scalar multiple of
$$
\tilde w_a^n :=  (r^n-r^{n+3})\circ \sigma_{a_-} - c_a (l^n-l^{n+3})\circ \sigma_a.
$$
where 
$$
c_a = \frac{a-a_-}{a_+ -a}.
$$

We next construct the short wavelets. Since $\breve k_a^0 = n$, $\breve
k_a^1 = n+3$ and from above $\bar k^0_a=\bar k^1_a=1$ it follows from Theorem \ref{waveletdim} that $\dim \breve W_a^n = 1$.
Also from the formula for $T_a^n$ above we find 
 \begin{align}
A_a^-&=\spam \{(r^n-r^{n+3})\circ \sigma_{a_-}\} \notag \\
A_a^+&=\spam \{(l^n-l^{n+3})\circ \sigma_a\} \notag
\end{align}
With the above results, using equation~(\ref{waveshort}) and another straightforward computation we see that $\breve W_a^n$ is spanned by
$$
\breve w_a^n :=\langle z^{n+3},r \rangle  \tilde \phi^{n+2} \circ \sigma_a -\langle \tilde \phi^{n+2},r \rangle z^{n+3} \circ \sigma_a.
$$

In summary  the wavelets constructed are,
\begin{align}\label{wavelets_approx_order}
\hat w_a^n&=(r^{n+3}-\frac{\langle  r^{n+3}  , r^{n+3} \rangle}{\langle  r^n  , r^n \rangle }r^n) \circ \sigma_{a_-}+ (l^{n+3}-\frac{\langle  r^{n+3}  , r^{n+3} \rangle}{\langle  r^n  , r^n \rangle }l^n) \circ \sigma_a \\
\tilde w_a^n&=(r^n-r^{n+3}) \circ \sigma_{a_-}-c_a (l^n-  l^{n+3}) \circ \sigma_a \notag \\
\breve w_a^n &=\langle z^{n+3},r \rangle  \tilde \phi^{n+2} \circ \sigma_a -\langle \tilde \phi^{n+2},r \rangle z^{n+3} \circ \sigma_a. \notag
\end{align}

\section{$\tau$-Wavelets}
\label{tauwavelets}

\subsection{Nested Knot Sequences Determined by $\tau$}

Let $\tau$ denote the "golden ratio" $\frac{1}{2}(1+\sqrt{5})$ which satisfies
the quadratic relation $\tau^{2}=1+\tau$.  A non-negative number $x$ is a \newterm{$\tau$-rational
number} if   it can be represented in the form $x=\sum\limits_{k=m}%
^{n}\varepsilon_{k}\tau^{k}$ where $m\leq n$ are integers  and each $\varepsilon_{k}\in\{0,1\}$.  Furthermore,  $x$ is  a \newterm{$\tau$-integer} if it has such a representation with $m=0$. By the above quadratic relation,
this representation   is unique if one
further requires that $\varepsilon_{k}\varepsilon_{k+1}=0$ whenever $m\leq
k\leq n-1$.  For example, the first few
non-negative $\tau$-integers are $0,1,\tau,\tau^{2},\tau^{2}+1,\tau
^{3},\ldots$. We denote the set of non-negative $\tau$-integers by
$\mathbf{Z}_{\tau}^{+}$. The unique representation of positive $\tau$-integers implies that $\mathbf{Z}_{\tau}^{+}\setminus\{0\}$ can be partitioned
into $\tau\mathbf{Z}_{\tau}^{+}\setminus\{0\}$ and $1+\tau^{2}\mathbf{Z}%
_{\tau}^{+}$. Further, $\tau\mathbf{Z}_{\tau}^{+}\setminus\{0\}$ can be
partitioned into $\tau+\tau^{2}\mathbf{Z}_{\tau}^{+}$ and $\tau^{2}+\tau
^{3}\mathbf{Z}_{\tau}^{+}$. Also, the difference between consecutive $\tau$-integers is either $1$ or $\frac{1}{\tau}$. More specifically, considering
$\mathbf{Z}_{\tau}^{+}$ as a knot sequence $\mathbf{a}$, if $a\in1+\tau
^{2}\mathbf{Z}_{\tau}^{+}$, then $a-a_{-}=1$ and $a_{+}-a=\frac{1}{\tau}$;
while if $a\in\tau+\tau^{2}\mathbf{Z}_{\tau}^{+}$, then $a-a_{-}=\frac{1}%
{\tau}$ and $a_{+}-a=1$; and finally if $a\in\tau^{2}+\tau^{3}\mathbf{Z}%
_{\tau}^{+}$, then $a-a_{-}=1$ and $a_{+}-a=1$. Letting $L$ denote the ``long''
difference of $1$ and letting $S$ denote the ``short'' difference of $\frac
{1}{\tau}$, we denote a non-zero $a\in\mathbf{a}$ as $LS$, respectively
$SL,LL$, if $a\in1+\tau^{2}\mathbf{Z}_{\tau}^{+}$, respectively $a\in\tau
+\tau^{2}\mathbf{Z}_{\tau}^{+}$, $a\in\tau^{2}+\tau^{3}\mathbf{Z}_{\tau}^{+}$.
(Note: The sequence of successive differences of elements of $\mathbf{Z}%
_{\tau}^{+}$ forms an infinite word,%
\[
f=LSLLSLSLLSLLSLSLLSLSL\ldots\text{,}%
\]
with alphabet $\{L,S\}$ which is invariant under the substitution $L\mapsto
LS$, $S\mapsto L$. $f$ is called the \textit{Fibonacci word}; see \cite{Lo}.)

For $k\in\mathbf{Z}$,   let $\mathbf{a}^{k}$ denote the knot sequence
$(\frac{1}{\tau})^{k}\mathbf{Z}_{\tau}^{+}$. The above unique representation
of $\tau$-integers shows that $\mathbf{a}^{k}\subset\mathbf{a}^{k+1}$ for
$k\in\mathbf{Z}$. Also, $\bigcup_{k\in\mathbf{Z}}\mathbf{a}^{k}$ is the
set of non-negative $\tau$-rational numbers which can be shown to be dense in
$\mathbf{R}^{+}$. For an integer $k$ and  $a\in\mathbf{a}^{k}$, let  $a_{-}^{k}$ (resp.
$a_{+}^{k}$) denote the predecessor (resp. successor) of $a$ relative to
$\mathbf{a}^{k}$. At level $k$, each interval $[a,a_{+}^{k}]$ has length
either $\frac{1}{\tau^{k}}$ or $\frac{1}{\tau^{k+1}}$. Such an interval will
be called long at level $k$ if it has length $\frac{1}{\tau^{k}}$ and short at
level $k$ if it has length $\frac{1}{\tau^{k+1}}$. The refinement from level
$k$ to level $k+1$ proceeds as follows. Each long interval at level $k$ is
split into two subintervals $[a,a_{+}^{k+1}]$ and $[a_{+}^{k+1},a_{+}^{k}]$,
where 
\begin{equation}\label{tausub}
a_{+}^{k+1}=(1-\frac{1}{\tau})a+\frac{1}{\tau}a_{+}^{k}.
\end{equation} It follows
that the left subinterval $[a,a_{+}^{k+1}]$ is long at level $k+1$, and the
right subinterval $[a_{+}^{k+1},a_{+}^{k}]$ is short at level $k+1$. Each
short interval at level $k$ is not subdivided and becomes long at level $k+1$;
i.e. if $a_{+}^{k}-a=\frac{1}{\tau^{k+1}}$ then $a_{+}^{k}=a_{+}^{k+1}$. 

\subsection{$\tau$-Wavelets of Haar}

We now show that the $\tau$-wavelets of Haar, constructed in \cite{GP}, can be
considered as a special case of the general wavelet construction outlined in Section  \ref{waveletConst}.
Let $\phi_{1}:=\chi_{\lbrack0,1]}$, $\phi_{2}:=\tau^{\frac{1}{2}}\chi
_{\lbrack0,\frac{1}{\tau}]}$, and $\Phi=\{\phi_{1}(\cdot-b)\mid b\in
\tau\mathbf{Z}_{\tau}^{+}\}\cup\{\phi_{2}(\cdot-c)\mid c\in1+\tau
^{2}\mathbf{Z}_{\tau}^{+}\}$. Then $\Phi$ is an orthonormal basis centered on
$\mathbf{Z}_{\tau}^{+}$ that is generated by  appropriate  $\tau$-integer
translations of $\phi_{1}$ and $\phi_{2}$ and  $S(\Phi)$ is the space of
piecewise constant functions in $L^{2}(\mathbf{R}^{+})$ with breakpoints in
$\mathbf{Z}_{\tau}^{+}$. For $k\in\mathbf{Z}$, let $\Phi
^{k}:=\{\tau^{\frac{k}{2}}f(\tau^{k}\cdot)\mid f\in\Phi\}$ ($=\{\tau^{\frac
{k}{2}}\phi_{1}(\tau^{k}\cdot-b)\mid b\in\tau\mathbf{Z}_{\tau}^{+}\}\cup
\{\tau^{\frac{k}{2}}\phi_{2}(\tau^{k}\cdot-c)\mid c\in1+\tau^{2}%
\mathbf{Z}_{\tau}^{+}\}$). Then $\Phi^{k}$ is an orthonormal basis centered on
$\mathbf{a}^{k}$ and  $S(\Phi^{k})$ is the space of piecewise constant
functions in $L^{2}(\mathbf{R}^{+})$ with breakpoints in $\mathbf{a}^{k}$.

Since $\mathbf{a}^{k}\subset\mathbf{a}^{k+1}$ for $k\in\mathbf{Z}$, it follows that   $S(\Phi^{k})\subset
S(\Phi^{k+1})$ for $k\in\mathbf{Z}$. In particular, $\Phi^{0}$ and
$\Phi^{1}$ are orthonormal bases centered on $\mathbf{Z}_{\tau}^{+}$ and
$S(\Phi^{0})\subset S(\Phi^{1})$. Theorem~\ref{ThmWavelets} thus applies to this
situation. We construct $\Psi$ as outlined in Section~\ref{waveletConst}. 

For $a>0 \in \mathbf{Z}_{\tau}^{+}$, it is easy to see that $\bar k_a^0 = \bar k_a^1 = m_a = m_a^{\pm}=0$. Also, $\breve k_a^0 =1$ for every $a \in \mathbf{Z}_{\tau}^{+}$. For $a \in \mathbf{Z}_{\tau}^{+}$, it follows from the manner in which $\mathbf{a}^0$ is refined to $\mathbf{a}^1$ that
$$
\breve k_a^1 =
\begin{cases}
1 & \text{if }a\in1+\tau^{2}\mathbf{Z}_{\tau}^{+}\\
2 & \text{if }a\in\tau\mathbf{Z}_{\tau}^{+}.
\end{cases}
$$
Thus Theorem \ref{waveletdim} implies that $\bar W_a = \{0\}$ for $a>0 \in \mathbf{Z}_{\tau}^{+}$, and 
$$
\dim \breve W_a =
\begin{cases}
0 & \text{if }a\in1+\tau^{2}\mathbf{Z}_{\tau}^{+}\\
1 & \text{if }a\in\tau\mathbf{Z}_{\tau}^{+}.
\end{cases}
$$
From equation (\ref{breveWa}), for $a \in \tau\mathbf{Z}_{\tau}^{+}$ we have $\breve W_a = (I-P_{\breve V_a^0})\breve V_a^1$. Also, for $a \in \tau\mathbf{Z}_{\tau}^{+}$, $\breve V_a^0$ is spanned by the function $\chi_{[a,a_+]}$ and $\breve V_a^1$ is spanned by the functions $\chi_{[a,a+\frac{1}{\tau}]}$ and $\chi_{[a+\frac{1}{\tau},a_+]}$. An easy computation shows that an orthonormal basis for $\breve W_a$ consists of the single function
$
\tau^{-\frac{1}{2}}\chi_{[a,a+\frac{1}{\tau}]}-\tau^{\frac{1}{2}}\chi_{[a+\frac{1}{\tau},a_+]}.
$

With
$$\psi:=\tau^{-\frac{1}{2}}\chi_{[0,\frac{1}{\tau}]}-\tau^{\frac{1}{2}}\chi_{[\frac{1}{\tau},1]},$$ it follows that $W = S(\Psi)$ where  $\Psi :=\{\psi(\cdot-b)\mid b\in\tau\mathbf{Z}_{\tau}^{+}\}$. Note that $\Psi$ is an
orthonormal basis centered on $\mathbf{Z}_{\tau}^{+}$, obtained by appropriate
$\tau$-integer translations of $\psi$.
 Letting $\Psi^{k}%
:=\{\tau^{\frac{k}{2}}\psi(\tau^{k}\cdot-b)\mid b\in\tau\mathbf{Z}_{\tau}%
^{+}\}$, for $k\in\mathbf{Z}$, it is easy to check that $\Psi^{k}$ is a basis
centered on $\mathbf{a}^{k}$ and that $S(\Phi^{k})\oplus S(\Psi
^{k})=S(\Phi^{k+1})$. Since $\bigcap\limits_{k\in\mathbf{Z}}S(\Phi
^{k})=\{0\}$ and $\bigcup\limits_{k\in\mathbf{Z}}S(\Phi^{k})$ is dense in
$L^{2}(\mathbf{R}^{+})$, it follows that $L^{2}(\mathbf{R}^{+})=\bigoplus
\limits_{k\in\mathbf{Z}}S(\Psi^{k})$, i.e. $\{\tau^{\frac{k}{2}}\psi
(\tau^{k}\cdot-b)\mid k\in\mathbf{Z,}b\in\tau\mathbf{Z}_{\tau}^{+}\}$ is an
orthonormal basis of $L^{2}(\mathbf{R}^{+})$. These functions are called the
$\tau$-wavelets of Haar in \cite{GP}.

\subsection{Continuous, Piecewise Quadratic $\tau$-Wavelets}
\label{cpqtw}
  
 Let 
 \begin{equation}\label{PhiDef}
 \Phi^k:=\Omega_{\mathbf{a}^k,\boldsymbol{\theta}},\qquad  k\in\Z
 \end{equation} denote the continuous, piecewise quadratic, orthonormal basis centered on $\mathbf{a}^k$ as described in Section~\ref{ex1} with constant parameter sequence $\boldsymbol{\theta}=\left(\theta_a\right)_{a\in \mathbf{a}^k}$
 where $\theta_a = 1/\tau$ for $a\in \mathbf{a}^k$ (here we  assume  that the components of $ \Phi^k$ have been normalized).    It follows from \eqref{tausub} that $(\mathbf{a}^k, \boldsymbol{\theta} )$   satisfy the hypotheses of Proposition~\ref{c0quadMRA}. Since, as previously discussed $\bigcup_k \mathbf{a}^k$ is dense in $J=\R_+$ and since
 $\boldsymbol{\alpha}=\bigcap_k \mathbf{a}^k\cup \mathbf{b}(\mathbf{a}^k,\boldsymbol{\theta}^k) =\{0\}$,  Proposition~\ref{c0quadMRA} implies that the spaces $V^k:=S(\Phi^k)$ form a {\bf multiresolution analysis of $L^2(\R_+)$}, that is, 
\begin{enumerate}
\item  $V^k\subset V^{k+1},\qquad (k\in \Z)$,
\item  $\bigcup_kV^k$ is dense in $L^2(\R_+)$,
\item $\bigcap_kV^k=\{0\}$.
 \end{enumerate}

As in the $\tau$-wavelets of Haar construction, the basis 
$
\Phi:=\Phi^0$
 centered on the knot
sequence $\mathbf{a}^{0}=\mathbf{Z}_{\tau}^{+}$ for $V_{0}$ can  also be
generated by the $\tau$-integer translations of a small number of functions
in $\Phi$.   Suppose $a$ and $a'$ in $\mathbf{Z}_{\tau}^{+}\setminus\{0\}$ have the same classification (say $LS$), 
then from the construction of $\Phi$  it follows that $\Phi_a$ is a translate of $\Phi_{a'}$. (Note: Here $\Phi_a$ stands for the more cumbersome, yet precise, $\Phi_{a,\mathbf{a}^0}^0$; see section \ref{sect:GeneralBases}.) Specifically, we have, for
$b\in \mathbf{Z}_{\tau}^{+}$, \begin{eqnarray}
\Phi_{1+\tau^{2}b}(\cdot) & = & \Phi_{1}(\cdot-\tau^{2}b), \nonumber \\
\Phi_{\tau+\tau^{2}b}(\cdot) & = & \Phi_{\tau}(\cdot-\tau^{2}b),   \label{tautrans} \\
\Phi_{\tau^{2}+\tau^{3}b}(\cdot) & = & \Phi_{\tau^{2}}(\cdot-\tau^{3}b).\nonumber
\end{eqnarray}
Thus every function in $\Phi$, with the exception of those in $\Phi_0$,  is a  $\tau$-integer translate of a
function from $\Phi_{1},$ $\Phi_{\tau},$ or $\Phi_{\tau^{2}}$.   For any positive  $a\in \mathbf{Z}_{\tau}^{+}$,
let 
$$\beta(a):=\begin{cases} 
1 & \text{if $a\in 1+\tau^2 \mathbf{Z}^+_\tau$,}\\
\tau & \text{if $a\in \tau+\tau^2 \mathbf{Z}^+_\tau$,}\\
\tau^2 & \text{if $a\in \tau^2+\tau^3 \mathbf{Z}^+_\tau,$}
\end{cases}
$$
and $\mu(a):=a-\beta(a)$. 
Then   $a=\beta(a)+\mu(a)$
where $\beta(a)\in \{1,\tau, \tau^2\}$ and 
$\mu(a)\in \tau^{2}\mathbf{Z}_{\tau}^{+}$ if $\beta(a)= 1$ or $\tau$ and $\mu(a)\in \tau^{3}\mathbf{Z}_{\tau}^{+}$ if $\beta(a)=  \tau^2$.
 Thus, we can write \eqref{tautrans} in the more compact form 
 \begin{equation}\label{Phia}
 \Phi_a=\Phi_{\beta(a)}(\cdot-\mu(a))\qquad (a\in \mathbf{Z}_{\tau}^{+}\setminus \{0\}).
 \end{equation}

\begin{figure}[thbp]
\begin{center}
\includegraphics[scale=.3]{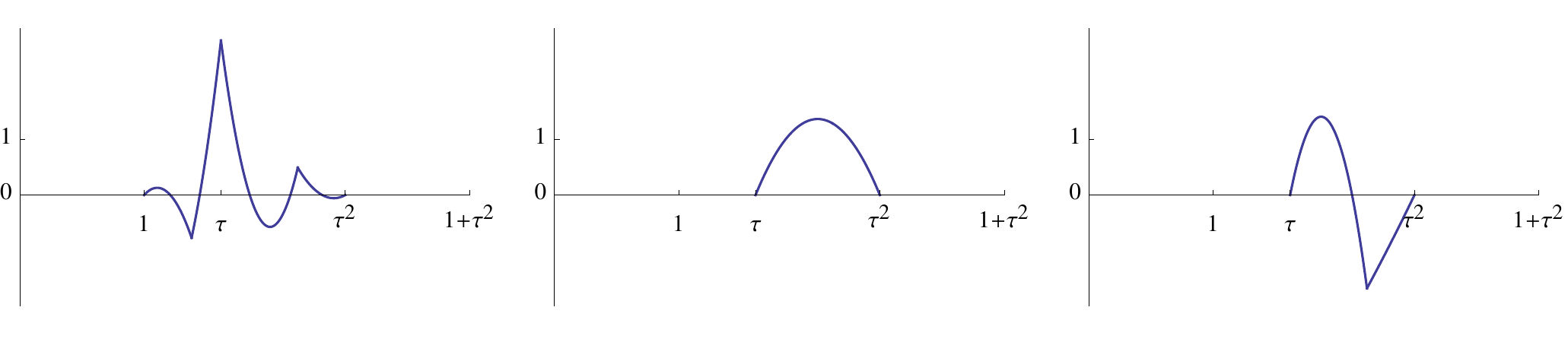}
\caption{$\Phi_{\tau}$.}
\label{phisubtau}
\end{center}
\end{figure}

For $k\in \mathbf{Z}$ and $a\in \mathbf{Z}_{\tau}^{+}$, let
$$\Phi_{k,a}=\tau^{k/2}\Phi_a(\tau^k\cdot).$$
We remark that with this definition, $\Phi_{k,a}=\Phi^k_{a/\tau^k,\mathbf{a}^k}$, and we apologize for the confusing notation. 
Note that $V^0 \subset V^1$ and both $\Phi^0$ and $\Phi^1$ can be regarded as orthonormal bases centered on the knot sequence $\mathbf{a}^0=\mathbf{Z}_{\tau}^{+}$. For $a\in \mathbf{a}^0$, $\Phi_a^1$ will denote the more precise $\Phi_{a,\mathbf{a}^0}^1$.
It follows that we have the following `polyphase' like representation:
$$
\Phi_0^1=
\left[
\begin{matrix}
\Phi_{1,0}\\
\Phi_{1,1}
\end{matrix}
\right],
\Phi_1^1=\Phi_{1,\tau},
\Phi_{\tau}^1=
\left[
\begin{matrix}
\Phi_{1,\tau^2}\\
\Phi_{1,1+\tau^2}
\end{matrix}
\right]
,~\text{and}~
\Phi_{\tau^2}^1=
\left[
\begin{matrix}
\Phi_{1,\tau^3}\\
\Phi_{1,1+\tau^3}
\end{matrix}
\right].
$$

Figure \ref{phisubtau} shows the three functions in $\Phi_{\tau}$ and Figure \ref{phisuper1subtau} shows the six functions in $\Phi_{\tau}^1$.

\begin{figure}[tbhp]
\begin{center}
\includegraphics[scale=.3]{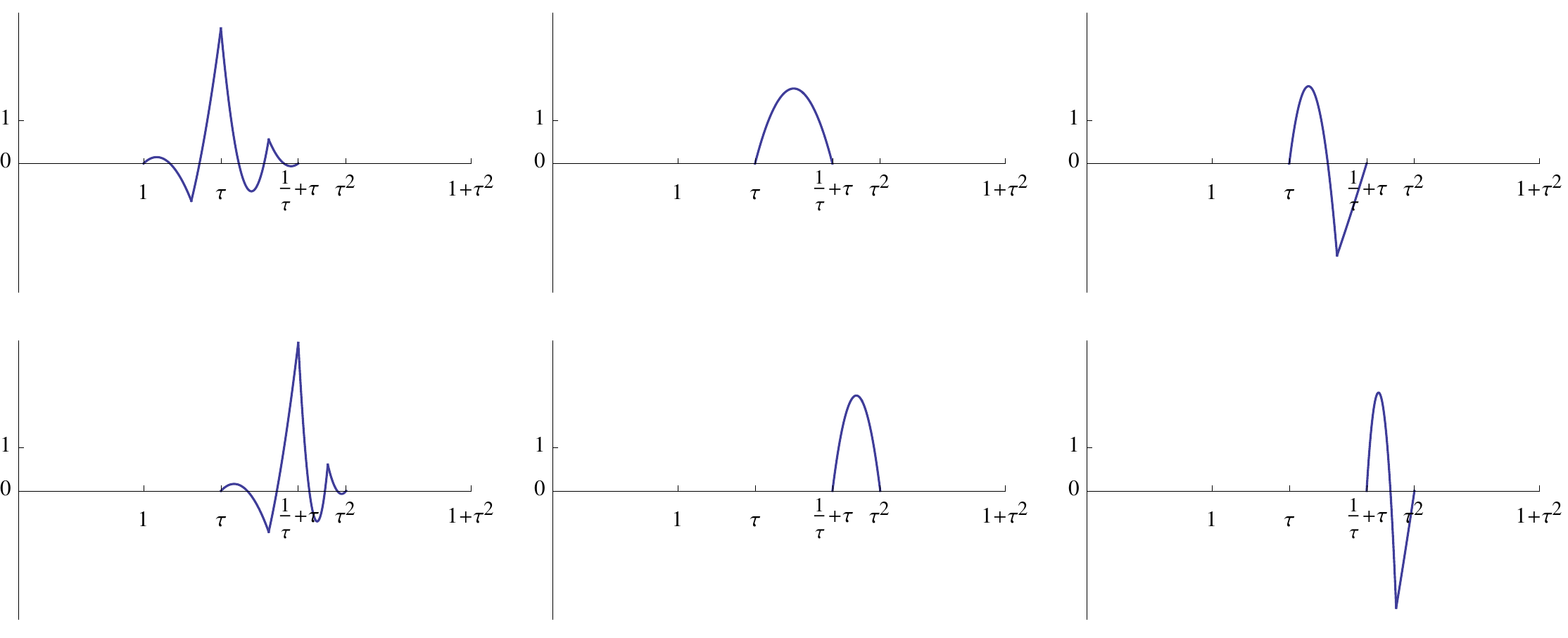}
\caption{$\Phi_{\tau}^1$.}
\label{phisuper1subtau}
\end{center}
\end{figure}

For $a\in \mathbf{Z}_{\tau}^{+}$ and $a>\tau^2$, 
$$
\Phi_a^1=\Phi_{\beta(a)}^1(\cdot - \mu(a)).
$$

The coefficient matrices from section \ref{MatCoeffs} can now be computed. For $a\in \mathbf{Z}_{\tau}^{+}$ and $a' \in \{a_-,a\}$, $c_{aa'}=\langle \Phi_a, \Phi_{a'}^1 \rangle$. In particular we have
$$
\Phi_0=c_{00}\Phi_0^1,~\text{and}~\Phi_a=c_{aa_-}\Phi_{a_-}^1+c_{aa}\Phi_{a}^1,~\text{for}~a>0.
$$
For any $k \in \mathbf{Z}$, $V^k \subset V^{k+1}$ and both $\Phi^k$ and $\Phi^{k+1}$ can be regarded as orthonormal bases centered on the knot sequence $\mathbf{a}^k$. For $a\in \mathbf{Z}_{\tau}^{+}$ it follows that  $\Phi_{a/\tau^k,\mathbf{a}^k}^{k+1}= \tau^{k/2}\Phi_a^1(\tau^k \cdot)$. We have, for $b \in \mathbf{a}^k$,
$$
\Phi_{0,\mathbf{a}^k}^k=c_{00}\Phi_{0,\mathbf{a}^k}^{k+1},~\text{and}~\Phi_{b,\mathbf{a}^k}^k = c_{(\tau^kb)(\tau^kb)_-}\Phi_{b_-,\mathbf{a}^k}^{k+1}+c_{(\tau^kb)(\tau^kb)}\Phi_{b,\mathbf{a}^k}^{k+1}~\text{for}~b>0.
$$
(Note: In the second equation above there are two occurrences  of the minus subscript. The first denotes the predecessor of $\tau^kb$ in the knot sequence $\mathbf{a}^0$. The second denotes the predecessor of $b$ in the knot sequence $\mathbf{a}^k$. See section \ref{sect:GeneralBases}.)

Recall that 
$$
\Phi_0^1=
\left[
\begin{matrix}
\Phi_{1,0}\\
\Phi_{1,1}
\end{matrix}
\right].
$$
Thus
$$
c_{00}=\langle \Phi_0,\Phi_0^1\rangle = 
\left[
\begin{matrix}
\langle \Phi_0,\Phi_{1,0}\rangle&\langle \Phi_0,\Phi_{1,1}\rangle
\end{matrix}
\right].
$$
Letting, 
$$
C_{a,a'}:=\langle \Phi_a,\Phi_{1,a'}\rangle, 
$$
for $a,a'\in \mathbf{Z}^+_\tau$,
we now have
$$
c_{00}=
\left[
\begin{matrix}
C_{0,0} & C_{0,1}
\end{matrix}
\right]
~\text{and}~
c_{10}=
\left[
\begin{matrix}
C_{1,0} & C_{1,1}
\end{matrix}
\right].
$$
and, similarly, $c_{11}=C_{1,\tau}$, $c_{\tau 1}=0$,
$$
c_{\tau \tau}=
\left[
\begin{matrix}
C_{\tau,\tau^2} & C_{\tau,1+\tau^2}
\end{matrix}
\right],
c_{\tau^2 \tau}=
\left[
\begin{matrix}
C_{\tau^2,\tau^2} & C_{\tau^2,1+\tau^2}
\end{matrix}
\right],
$$
and
$$
c_{\tau^2 \tau^2}=
\left[
\begin{matrix}
C_{\tau^2,\tau^3} & C_{\tau^2,1+\tau^3}
\end{matrix}
\right].
$$
For $a>\tau^2$, $c_{aa}=c_{\beta(a) \beta(a)}$ and $c_{aa_-}=c_{\beta(a) \beta(a)_-}$.

In Table~\ref{cTable}  we provide (with the aid of {\em Mathematica}\texttrademark) the matrices $C_{a,a'}$ defined above.  Here we assume that $\Phi_a$ is ordered as follows: of the three components,  the second and third components are the ``$q$'' and  ``$z$''
components from $\breve\Phi_a$, respectively, while the first component is the component of $\Phi_a$ that does not vanish at $a$.

{\tiny \ctable }

\bigskip

We now construct an orthogonal wavelet basis $\Psi$, centered on
$\mathbf{a}^0$,  for the continuous, piecewise quadratic scaling functions
$\Phi$ such that $$S(\Phi^0)\oplus S(\Psi)=S(\Phi^1)$$  using the method
from section \ref{waveletConst}. As mentioned above this  provides an
example where the spaces $T^{\pm}_a$ contain non-zero functions.

Following Section \ref{waveletConst}, we determine the dimensions $\bar k_a^{\epsilon},  m_a, \text{and}~ m_a^{\pm}$.
If $a>0$ then $\bar k_a^0 = \bar k_a^1 = 1$ and $m_a=0$. Also, for $a>0$,
\begin{equation} \label{taum+}
m_a^+ =
\begin{cases}
1 & \text{if}~a \in 1+\tau^2 \Z_{\tau}^+\\
0 & \text{otherwise}
\end{cases}
\end{equation}
and 
\begin{equation} \label{taum-}
m_a^- =
\begin{cases}
1 & \text{if}~a \in \tau+\tau^2 \Z_{\tau}^+\\
0 & \text{otherwise}.
\end{cases}
\end{equation}
From Theorem \ref{waveletdim} we get
$$
\dim \bar W_a=
\begin{cases}
1 & \text{if either}~a \in 1+\tau^2 \Z_{\tau}^+ ~\text{or}~a \in \tau+\tau^2 \Z_{\tau}^+\\
2 & \text{if}~a \in \tau^2+\tau^3 \Z_{\tau}^+.
\end{cases}
$$
It follows that for $a \in 1+\tau^2 \Z_{\tau}^+ ~\text{or}~a \in \tau+\tau^2 \Z_{\tau}^+$ that $\bar W_a= \hat W_a$ which is spanned by a single function. For $a \in \tau^2+\tau^3 \Z_{\tau}^+$, a basis for $\bar W_a$ consists of two functions, one forming a basis for $\hat W_a$ and the other a basis for $\tilde W_a$.
In all cases, Lemma \ref{stepone} implies that for $a>0$, $\dim \hat W_a=1$. To find a basis for 
$\hat W_a$ we begin by constructing bases for $\hat W_1,\hat W_{\tau}$, and $\hat W_{\tau^2}$. We use the notation introduced in Section \ref{ex1}.

For $a>0$ an orthonormal basis of $\bar V_a^0$ consists of the single function $\bar \phi_a^0$ which is a suitable scalar multiple of $r^{1/\tau} \circ \sigma_{a_-}+l^{1/\tau} \circ \sigma_a.$ It follows that an orthonormal basis of $\bar V_1^1$ consists of a single function $\bar \phi_1^1$ which is a suitable scalar multiple of $\bar \phi_{\tau}^0(\tau \cdot)$. Similarly an orthonormal basis of $\bar V_{\tau}^1$ consists of a single function $\bar \phi_{\tau}^1$ which is a suitable scalar multiple of $\bar \phi_{\tau^2}^0(\tau \cdot)$, and an orthonormal basis of $\bar V_{\tau^2}^1$ consists of a single function $\bar \phi_{\tau^2}^1$ which is a suitable scalar multiple of $\bar \phi_{\tau^3}^0(\tau \cdot)$.

For $a=1,\tau,\tau^2$, let
$$
\hat w_a := c_{\hat w_a}(\bar \phi_a^1- \langle \bar \phi_a^1,\bar \phi_a^0 \rangle \bar \phi_a^0),
$$
where  $c_{\hat w_a}$ is a constant so the $\Vert \hat w_a\Vert=1$.
By construction for each $a>0$
 \begin{equation}
 \hat w_a=\hat w_{\beta(a)}\big(\cdot-\mu(a)\big)
 \end{equation}
forms an orthonormal basis for the one dimensional  space $\hat W_a$.
 Figure~\ref{tauhatwavelets} shows the graphs of $\hat w_1,\hat w_\tau$, and $\hat w_{\tau^2}$.

\begin{figure}[htbp]
\begin{center}
\includegraphics[scale=1]{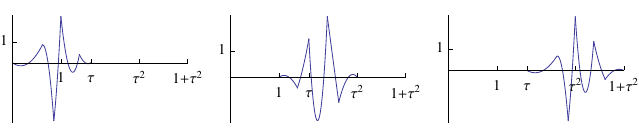}
\caption{$\hat w_a$ for $a=1,\tau,\tau^2$}
\label{tauhatwavelets}
\end{center}
\end{figure}

Next we construct a basis for $\tilde W_{\tau^2}$.  We begin by constructing the spaces defined in Section \ref{dws}. A basis for $T_{\tau^2}$ consists of the single function
$$
t_{\tau^2}:=\bar \phi_{\tau^2}^0- \langle \bar \phi_{\tau^2}^0, \bar \phi_{\tau^2}^1 \rangle \bar \phi_{\tau^2}^1.
$$
Since $m_{\tau^2}=m_{\tau^2}^+=m_{\tau^2}^-=0$,  by equations (\ref{taum+}) and (\ref{taum-}), it follows from Lemma \ref{ydim} that $U_{\tau^2}=T_{\tau^2}$. Thus $S_{\tau^2}$ is spanned by the single function
$$
s_{\tau^2}:= (\chi_{[\tau^2,1+\tau^2]}-\chi_{[\tau,\tau^2]})t_{\tau^2},
$$
and so, by
 Lemma \ref{tildew}, an orthonormal basis of $\tilde W_{\tau^2}$ consists of the single function
$$
\tilde w_{\tau^2}:= c_{\tilde w_{\tau^2}}(s_{\tau^2}- \langle s_{\tau^2},
\bar \phi_{\tau^2}^0 \rangle \bar \phi_{\tau^2}^0  -  \langle s_{\tau^2},
\hat w_{\tau^2} \rangle \hat w_{\tau^2}),
$$
where $c_{\tilde w_{\tau^2}}$ is chosen so that $\tilde w_{\tau^2}$ is of
norm one.
It follows that for $b \in \Z_{\tau}^+$, an orthonormal basis for $\tilde W_{\tau^2 + b \tau^3}$ consists of the single function
\begin{equation}
\tilde w_{\tau^2 + b \tau^3} = \tilde w_{\tau^2}(\cdot -  b \tau^3).
\end{equation}

In Figure~\ref{tautildewavelet} we see $\tilde w_{\tau^2}$.

\begin{figure}[htbp]
\begin{center}
\includegraphics[scale=1]{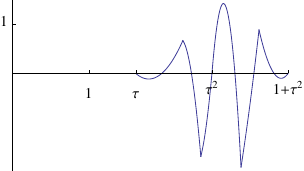}
\caption{$\tilde w_{\tau^2}$}
\label{tautildewavelet}
\end{center}
\end{figure}

Completing the wavelet construction, we next determine $\breve W_a$ for $a \ge 0$. Observe that
$$
\breve k_a^0=
\begin{cases}
3 &\text{if}~a=0\\
2 &\text{if}~a>0
\end{cases}
$$
and 
$$
\breve k_a^1=
\begin{cases}
6 &\text{if}~a=0\\
2 &\text{if}~a \in 1+\tau^2 \Z_{\tau}^+\\
5 & \text{if}~a >0 \in \tau \Z_{\tau}^+.
\end{cases}
$$

It follows from Theorem \ref{waveletdim} and equations (\ref{taum+}) and (\ref{taum-}) that
$$
\dim \breve W_a =
\begin{cases}
2 &\text{if}~a=0\\
0 &\text{if}~a \in 1+\tau^2 \Z_{\tau}^+\\
1 &\text{if}~a  \in \tau \Z_{\tau}^+\setminus\{0\}.
\end{cases}
$$
We construct the spaces $\breve W_a$ using equation (\ref{breveWa}). Note that for any $a>0$, $A_a^-$ is spanned by the single function
$$
f_a^-:=  (\bar \phi_a^0-\langle \bar \phi_a^0, \bar \phi_a^1 \rangle \bar \phi_a^1)\chi_{[a_-,a]},
$$
and $A_a^+$ is spanned by the single function
$$
f_a^+:=(\bar \phi_a^0-\langle \bar \phi_a^0, \bar \phi_a^1 \rangle \bar \phi_a^1) \chi_{[a,a_+]}.
$$

 We first observe that $\bar \phi_{1+\tau^2}^0(\tau \cdot) \in \breve V_{\tau}^1$. Also, $\breve V_{\tau}^0$ is spanned by the functions $q \circ \sigma_{\tau}$ and $z^{1/\tau} \circ \sigma_{\tau}$. Let $$\breve w_{\tau}:=\breve c_{\tau}(I-P_{\spam\{f_{\tau}^+, q \circ \sigma_{\tau}, z^{1/\tau} \circ \sigma_{\tau}, f_{\tau^2}^-\}})\bar \phi_{1+\tau^2}^0(\tau \cdot),$$  where $\breve c_{\tau}$ is a normalization constant.  Since dim $\breve W_{\tau}=1$,  it follows that $\breve w_{\tau}$ forms an orthonormal basis for $\breve W_{\tau}$.
Furthermore,    $\breve w_{\tau^2}:= \breve w_{\tau}(\cdot-1)$   forms an orthonormal basis for $\breve W_{\tau^2}$.

For $b \in \Z_{\tau}^+$, we then define
\begin{align*}
\breve w_{\tau + \tau^2b}&:= \breve w_{\tau}(\cdot - \tau^2b)\\
\breve w_{\tau^2 + \tau^3b}&:= \breve w_{\tau^2}(\cdot - \tau^3b).
\end{align*}
It follows that for $b \in \Z_{\tau}^+$, $\breve w_{\tau + \tau^2b}$ (resp. $\breve w_{\tau^2 + \tau^3b}$ ) forms an orthonormal basis for     $\breve W_{\tau + \tau^2b}$ (resp. $\breve W_{\tau^2 + \tau^3b}$).
 Figure~\ref{taubrevewavelet} shows the graph of $\breve w_{\tau}$.

\begin{figure}[htbp]
\begin{center}
\includegraphics[scale=1]{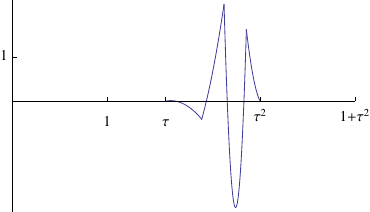}
\caption{$\breve w_{\tau}$}
\label{taubrevewavelet}
\end{center}
\end{figure}

Recall that $\breve W_0 = \breve V_0^1 \ominus (\breve V_0^0 \oplus A_1^-)$ and that  $\dim \breve W_0=2$.   
It is   easy to check that $\breve w_{0,1}:= \breve w_{\tau}(\cdot +\tau)$ is an element of $\breve W_0$. Then   $\breve w_{0,2}$ is chosen to be the unique (up to a sign) element in $\breve W_0 $ so that $\{\breve w_{0,1},\breve w_{0,2}\}$ is an orthonormal basis for $\breve W_0$.
 Figure~\ref{taubrevewavelet02} shows the graph of $\breve w_{0,2}$.

\begin{figure}[htbp]
\begin{center}
\includegraphics[scale=1]{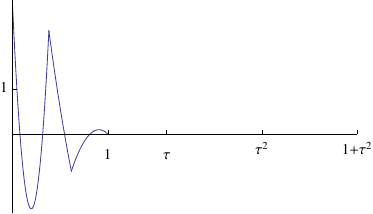}
\caption{$\breve w_{0,2}$}
\label{taubrevewavelet02}
\end{center}
\end{figure}

We now have that
$$
\Psi_0=
\left[
\begin{matrix}
\breve w_{0,1}\\
\breve w_{0,2}
\end{matrix}
\right],
\Psi_1=
\left[
\begin{matrix}
\hat w_1
\end{matrix}
\right],
\Psi_{\tau}=
\left[
\begin{matrix}
\hat w_{\tau}\\
\breve w_{\tau}
\end{matrix}
\right],
~\text{and}~
\Psi_{\tau^2}=
\left[
\begin{matrix}
\hat w_{\tau^2}\\
\tilde w_{\tau^2}\\
\breve w_{\tau^2}
\end{matrix}
\right].
$$
For $a>\tau^2$, it follows from \eqref{Phia} that $\Psi_a=\Psi_{\beta(a)}(\cdot - \mu(a))$.

Since $S(\Phi^k)\subset S(\Phi^{k+1})$, for $k\in \mathbf{Z}$, it follows from Theorem \ref{ThmWavelets} that there exists an orthonormal basis, $\Psi^k$ centered on $\mathbf{a}^k$ so that $S(\Phi^{k+1})=S(\Phi^k) \oplus S(\Psi^k)$.
For $k\in \mathbf{Z}$ and $a\in \mathbf{Z}_{\tau}^+$, let
$$
\Psi_{k,a}:=\tau^{k/2}\Psi_a(\tau^k\cdot).
$$
Then, for $k\in \mathbf{Z}$ and $b\in \mathbf{a}^k$,
$$
\Psi_{b,\mathbf{a}^k}^k=\Psi_{k,\tau^kb}.
$$

The wavelet coefficient matrices from section \ref{MatCoeffs} can now be computed. For $a\in \mathbf{Z}_{\tau}^{+}$ and $a' \in \{a_-,a\}$, $d_{aa'}=\langle \Psi_a, \Phi_{a'}^1 \rangle$. In particular we have
$$
\Psi_0=d_{00}\Phi_0^1,~\text{and}~\Psi_a=d_{aa_-}\Phi_{a_-}^1+d_{aa}\Phi_{a}^1,~\text{for}~a>0.
$$

Recall that 
$$
\Phi_0^1=
\left[
\begin{matrix}
\Phi_{1,0}\\
\Phi_{1,1}
\end{matrix}
\right].
$$
Thus
$$
d_{00}=\langle \Psi_0,\Phi_0^1\rangle = 
\left[
\begin{matrix}
\langle \Psi_0,\Phi_{1,0}\rangle&\langle \Psi_0,\Phi_{1,1}\rangle
\end{matrix}
\right].
$$
This suggests defining, for $a,a'\in \mathbf{Z}^+_\tau$,
$$
D_{a,a'}:=\langle \Psi_a,\Phi_{1,a'}\rangle.
$$
We now have
$$
d_{00}=
\left[
\begin{matrix}
D_{0,0} & D_{0,1}
\end{matrix}
\right]
~\text{and}~
d_{10}=
\left[
\begin{matrix}
D_{1,0} & D_{1,1}
\end{matrix}
\right].
$$
Similarly it follows that $d_{11}=D_{1,\tau}$, $d_{\tau 1}=0$,
$$
d_{\tau \tau}=
\left[
\begin{matrix}
D_{\tau,\tau^2} & D_{\tau,1+\tau^2}
\end{matrix}
\right],
d_{\tau^2 \tau}=
\left[
\begin{matrix}
D_{\tau^2,\tau^2} & D_{\tau^2,1+\tau^2}
\end{matrix}
\right],
$$
and
$$
d_{\tau^2 \tau^2}=
\left[
\begin{matrix}
D_{\tau^2,\tau^3} & D_{\tau^2,1+\tau^3}
\end{matrix}
\right].
$$
For $a>\tau^2$, $d_{aa}=d_{\beta(a) \beta(a)}$ and $d_{aa_-}=d_{\beta(a) \beta(a)_-}$.
In Table~\ref{dTable}  we provide the matrices $D_{a,a'}$, mentioned above, which were computed with the aid of {\em Mathematica}\texttrademark.

{\tiny
\dtable
}

\newpage

\noindent
{\bf Acknowledgements:}  We thank the referees for their careful reading and thoughtful suggestions.

\end{document}